\documentclass{amsart}
\usepackage[utf8]{inputenc}

\usepackage{fullpage}
\usepackage{tabularx}
\usepackage{arydshln, tikz, amsfonts}
\usepackage{graphicx}
\usepackage[mathcal]{euscript}
\usepackage{float}
\usepackage{color}
\usepackage{tikz}
\usepackage{verbatim}
\usepackage{pdfpages}
\usepackage{enumerate}
\usepackage{amsmath}
\usepackage{subfiles}
\usepackage{multirow}

\newtheorem{theorem}{Theorem}[section]
\newtheorem*{theorem-non}{Theorem}
\newtheorem{lemma}[theorem]{Lemma}
\newtheorem{corollary}[theorem]{Corollary}

\theoremstyle{definition}
\newtheorem{definition}[theorem]{Definition}
\newtheorem{example}[theorem]{Example}

\theoremstyle{remark}
\newtheorem{remark}[theorem]{Remark}

\numberwithin{equation}{section}

\newcommand{\N}{\mathbb{N}}

\newcommand{\R}{\mathbb{R}}

\allowdisplaybreaks

\newcommand{\Var}{\operatorname{Var}}

\definecolor{blue}{rgb}{0,0,1}
\def\blue{\color{blue}}
\def\adam{\color{red} Adam: }

\title{The distribution of genera of 2-bridge knots}
\author{Moshe Cohen, Abigail  DiNardo, Adam M. Lowrance, Steven Raanes, Izabella M. Rivera, Andrew J. Steindl, Ella S. Wanebo}

\begin{document}

\begin{abstract}
The average genus of a 2-bridge knot with crossing number $c$ approaches $\frac{c}{4} + \frac{1}{12}$ as $c$ approaches infinity, as proven by Suzuki and Tran and independently Cohen and Lowrance.  In this paper, for the genera of $2$-bridge knots of a fixed crossing number $c$, we show that the median and mode are both $\lfloor \frac{c+2}{4} \rfloor$ and that the variance approaches $\frac{c}{16}-\frac{17}{144}$ as $c$ approaches infinity.  We prove that the distribution of genera of 2-bridge knots is asymptotically normal. 
\end{abstract}

%


\maketitle

\section{Introduction}
\label{sec:Intro}
The genus $g(K)$ of a knot $K$ in $S^3$ is the minimum genus of any Seifert surface of $K$, where a Seifert surface is an oriented surface in $S^3$ whose boundary is $K$. Dunfield et al. \cite{Dun:knots} produced experimental data suggesting that the genus and crossing number of a knot are linearly related. Baader, Kjuchukova, Lewark, Misev, and Ray \cite{BKLMR} showed that the average genus of a 2-bridge knot with crossing number $c$ is bounded from below by $\frac{c}{4}$ and also showed that the quotient of the average $4$-genus and the average genus of 2-bridge knots of crossing number $c$  goes to zero as $c$ goes to infinity. Cohen \cite{CohenBound} found another lower bound for the average genus of a 2-bridge knots. Suzuki and Tran \cite{SuzTran} and independently Cohen and Lowrance \cite{CoLow} found an exact formula for the average genus of the set of $2$-bridge knots with crossing number $c$ and proved that this average genus asymptotically approaches $\frac{c}{4} + \frac{1}{12}$ as $c$ goes to infinity. See also similar work by Ray and Diao \cite{RayDiao}.

Let $\mathcal{K}_c$  be the set of 2-bridge knots with crossing number $c$, where only one of a knot and its mirror image is in the set $\mathcal{K}_c$. For example, only one of the right-handed or left-handed trefoil is in $\mathcal{K}_3$, and $|\mathcal{K}_3|=1$. Since a knot has the same genus as its mirror image, our results do not depend on this choice.  Define the random variable $G_c:\mathcal{K}_c\to\R$ by $G_c(K)=g(K)$, the genus of $K$. We call $G_c$ the \textit{genus of 2-bridge knots with $c$ crossings} random variable, or just the \textit{genus random variable} for short. With this notation, the result from the previous paragraph on average genus can be stated as the expected value $E(G_c)$ of $G_c$ goes to $\frac{c}{4}+\frac{1}{12}$ as $c$ goes to infinity.

In our first main result, we compute the median and mode of $G_c$ and find the limit variance of $G_c$ as $c$ goes to infinity. In Theorem \ref{thm:variance}, we give an exact but unwieldy formula for the variance $\Var(G_c)$ of $G_c$.

\begin{theorem}
\label{thm:main}
Let $G_c$ be the genus random variable.
\begin{enumerate}
\item The median of $G_c$ is $\left\lfloor\frac{c+2}{4}\right\rfloor$, and the mode of $G_c$ is $\left\lfloor\frac{c+2}{4}\right\rfloor$.
\item The variance $\Var(G_c)$ of $G_c$ approaches $\frac{c}{16} - \frac{17}{144}$ as $c\to\infty$.
\end{enumerate}
\end{theorem}

In our next main result, we show that the probability distribution of the genera of 2-bridge knots is asymptotically normal. For a real-valued random variable $X$ and a real number $x\in\mathbb{R}$, the cumulative distribution function of $X$ is the probability that $X$ is less than or equal to $x$, denoted by $P(X\leq x)$. The cumulative distribution function of a normal distribution with mean $\mu$ and standard deviation $\sigma$ is
\[\Phi_{\mu,\sigma}(x) = \int_{-\infty}^x \frac{1}{\sigma\sqrt{2\pi}}e^{-\frac{1}{2}\left(\frac{t-\mu}{\sigma}\right)^2}\; dt.\]

\begin{theorem}
    \label{thm:normal}
    Let $G_c$ be the genus random variable, and let $n=\left\lfloor\frac{c-3}{2}\right\rfloor$. For any $x\in\mathbb{R}$, the difference between the cumulative distribution functions of the genus random variable and the normal random variable with mean $\mu =\frac{n}{2}$ and standard deviation $\sigma=\frac{\sqrt{n}}{2}$ approaches zero as $c$ goes to infinity, that is, 
    \[\lim_{c\to\infty}P(G_c\leq x) - \Phi_{\mu,\sigma}(x) = 0.\]
\end{theorem}

In several other contexts, the expected value and distribution of genera have been studied, leading to similar results about normal distributions. Brooks and Makover constructed a model for random Riemann surfaces \cite{BroMak}, and Gamburd and Makover computed the expected genus of a random Riemann surface \cite{BroMak2}.   Linial and Nowik give the expected genus of a random chord diagram \cite{LinNow}.  Chmutov and Pittel proved that the distribution of genera of surfaces obtained by randomly gluing the sides of an $n$-gon is asymptotically normal \cite{CP1} and the distribution of genera of surfaces obtained by gluing the sides of multiple polygonal disks together is also asymptotically normal \cite{CP2}. Even-Zohar and Farber \cite{EZF} studied surfaces with boundary obtained from a collection of polygonal disks by gluing some of their sides together and showed that their genera is asymptotically a bivariate normal distribution.  Shrestha \cite{Shr:genus} studied similar square-tiled surfaces and showed that their genera satisfy a local central limit theorem.

In our final main result, we express the number of 2-bridge knots of a particular genus and crossing number as a sum. Define $\bar{t}(c,g)$ to be the number of 2-bridge knots with genus $g$, that is define $\bar{t}(c,g) = |\{K\in\mathcal{K}_c~:~g(K)=g\}|$.
 
\begin{theorem}
\label{thm:tbarformula}
The number of $2$-bridge knots with crossing number $c$ and genus $g$ is
\[ \bar{t}(c,g) = \frac{1}{2}\left((-1)^{c'-g-1} \sum_{n=0}^{c'-g-1}(-1)^n{n+g-1\choose n} + (-1)^{c-1}\sum_{n=0}^{c-2g-1}(-1)^{n}{n+2g-1 \choose n}\right),\]
where $c'=\left\lfloor\frac{c+1}{2}\right\rfloor$ and $1\leq g \leq \left\lfloor\frac{c-1}{2}\right\rfloor$. If $g>\left\lfloor\frac{c-1}{2}\right\rfloor$, then $\bar{t}(c,g)=0$.
\end{theorem}



Our approach to proving these theorems is similar to the approach of Cohen and Lowrance \cite{CoLow} and uses Koseloff and Pecker's billiard table model for 2-bridge knots \cite{KP1, KP2}. Cohen and Krishnan \cite{CoKr} and Cohen, Even-Zohar, and Krishnan \cite{CoEZKr} showed that the billiard table model for 2-bridge knots results in an explicit enumeration of the elements of $\mathcal{K}_c$, as well as an alternating diagram for each knot in $\mathcal{K}_c$. The explicit list of $2$-bridge knots allows us to argue inductively, and the associated alternating diagrams make genus computations possible. 

\indent This paper is organized as follows. In Section \ref{sec:Background}, we detail how to use billiard diagrams to generate all 2-bridge knots. In Section \ref{sec:recursions}, we find recursive rules for the number of 2-bridge knots with crossing number $c$ and genus $g$ and find an explicit formula for $\bar{t}(c,g),$ proving Theorem \ref{thm:tbarformula}. In Section \ref{sec:Mode}, we find a formula to show the median and mode for a fixed $c$, proving part (1) of Theorem \ref{thm:main}. In Section \ref{sec:variance}, we find a formula to calculate the variance of the distribution, proving part (2) of Theorem \ref{thm:main}. In Section \ref{sec:normal}, we prove Theorem \ref{thm:normal}, showing that the genus of $2$-bridge knots of a fixed crossing number is asymptotically normal. 

\medskip

{\bf Acknowledgements.} This paper is the result of a summer research project in the Undergraduate Research Science Institute at Vassar College. The third author is supported by NSF grant DMS-1811344.

\section{Background}
\label{sec:Background}


We generate the set $\mathcal{K}_c$ of 2-bridge knots with a fixed crossing number $c$ using billiard table trajectories developed by Koseleff and Pecker \cite{KP1, KP2}. A \textit{billiard table diagram} of a 2-bridge knot is drawn as follows. First, draw a $3$ by $b$ grid where $3$ and $b$ are coprime. Start at the bottom left corner of the rectangle and draw a line segment such that it makes a $45$ degree angle with both the bottom and left edge of the rectangle. The segment continues until it hits an edge of the rectangle, at which point we draw a new segment whose starting point is the ending point of the previous segment. The two segments should form a right angle. Repeat this process until a line segment hits a corner, at which point it terminates. The resulting piecewise-linear arc is a \textit{billiard table trajectory} and intersects itself $b-1$ times, once on each vertical grid line except the two that are edges of the rectangle. We turn the billiard table trajectory into a billiard table diagram by making each self intersection 
$\tikz[baseline=.6ex, scale = .4]{
\draw (0,0) -- (1,1);
\draw (0,1) -- (1,0);
}
~$
into a crossing in one of two ways: 
$\tikz[baseline=.6ex, scale = .4]{
\draw (0,0) -- (1,1);
\draw (0,1) -- (.3,.7);
\draw (.7,.3) -- (1,0);
}
~$, denoted $+$, or
 $\tikz[baseline=.6ex, scale = .4]{
\draw (0,0) -- (.3,.3);
\draw (.7,.7) -- (1,1);
\draw (0,1) -- (1,0);
}
~$, denoted $-$, and joining the two ends that intersect corners. A billiard table trajectory and its corresponding billiard table diagram are shown in Figure \ref{fig:billiard}. In the billiard table diagram, the two ends of the line at the corners are understood to be connected such that they form a knot. With a height function where the leftmost points are maxima, the resulting diagram has two bridges and so represents the unknot or a 2-bridge knot.

\begin{figure}[!h]
\centering

\begin{tikzpicture}[scale=.55]
\draw[dashed, white!50!black] (0,0) rectangle (11,3);
\foreach \x in {1,...,10}
	{\draw[dashed, white!50!black] (\x,0) -- (\x,3);}
\foreach \x in {1,2}
	{\draw[dashed, white!50!black] (0,\x) -- (11, \x);}
\foreach \x in {0,2,4,6}
	{\draw[thick] (\x,0) -- (\x+3,3);
	\draw[thick] (\x+1,3) -- (\x+4,0);}
\draw[thick] (1,3) -- (0,2) -- (2,0);
\draw[thick] (9,3) -- (11,1) -- (10,0);
\draw[thick] (8,0) -- (11,3);
\draw[thick, ->] (0,0) -- (1.5,1.5);

\begin{scope}[xshift = 12 cm]
	\draw[dashed, white!50!black] (0,0) rectangle (11,3);
	\foreach \x in {1,...,10}
		{\draw[dashed, white!50!black] (\x,0) -- (\x,3);}
	\foreach \x in {1,2}
		{\draw[dashed, white!50!black] (0,\x) -- (11, \x);}
	
	\draw[thick] (0,0) -- (1.8,1.8);
	\draw[thick] (2.2,2.2) -- (3,3) -- (3.8,2.2);
	\draw[thick] (4.2,1.8) -- (6,0) -- (7.8,1.8);
	\draw[thick] (8.2,2.2) -- (9,3) -- (9.8,2.2);
	\draw[thick] (10.2,1.8) -- (11,1) -- (10,0) -- (7,3) -- (5.2,1.2);
	\draw[thick] (4.8,.8) -- (4,0) -- (1,3) -- (0,2) -- (0.8,1.2);
	\draw[thick] (1.2,0.8) -- (2,0) -- (2.8,0.8);
	\draw[thick] (3.2,1.2) -- (5,3) -- (5.8,2.2);
	\draw[thick] (6.2,1.8) -- (6.8,1.2);
	\draw[thick] (7.2,0.8) -- (8,0) -- (8.8,0.8);
	\draw[thick] (9.2,1.2) -- (11,3);
	\draw[thick, ->] (0,0) -- (1.5,1.5);
	\end{scope}
	\end{tikzpicture}

\caption{A $3$ by $11$ billiard table trajectory and its billiard table diagram corresponding to the word $+--+-++--+$.}
    \label{fig:billiard}
\end{figure}
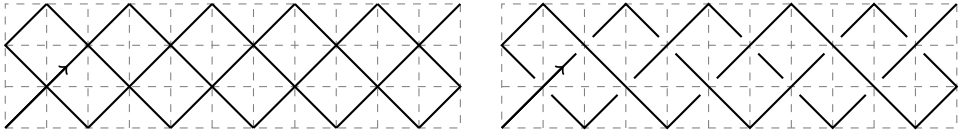

All the information of the billiard table diagram is captured by its associated sequence of pluses and minuses, which we think of as a word comprised of the symbols $\{+,-\}$. Since any 2-bridge knot can be generated by an infinite number of words, as in Cohen and Krishnan \cite{CoKr} and Cohen, Even-Zohar, and Krishnan \cite{CoEZKr}, we restrict the set of words so that each 2-bridge knot will only be counted exactly once or exactly twice, as stated in Theorem \ref{thm:list} below.

\begin{definition}
\label{def:TC}
Define $T(c)$ to be the \textit{partially double counted set of words corresponding to 2-bridge knots with crossing number $c$}. The elements of this set are words in the symbols $\{+,-\}$  corresponding to billiard table diagrams. When $c$ is odd, then a word $w$ is in $T(c)$ if and only if 
\[
w=(+)^{\varepsilon_1}(-)^{\varepsilon_2}(+)^{\varepsilon_3}(-)^{\varepsilon_4}\ldots(-)^{\varepsilon_{c-1}}(+)^{\varepsilon_c}, \]
and when $c$ is even, then a word $w$ is in $T(c)$ if and only if 
\[w=(+)^{\varepsilon_1}(-)^{\varepsilon_2}(+)^{\varepsilon_3}(-)^{\varepsilon_4}\ldots(+)^{\varepsilon_{c-1}}(-)^{\varepsilon_c},\]
where in both cases $\varepsilon_i\in\{1,2\}$ for $i\in\{1,\ldots,c\}$, $\varepsilon_1=\varepsilon_c=1$, and the length of the word $\ell=\sum_{i=1}^{c}\varepsilon_i \equiv 1$ mod $3$. Each subword $(+)^{\varepsilon_i}$ or $(-)^{\varepsilon_i}$ is called a \textit{run} in $w$.
\end{definition}

Although the billiard table diagram associated with a word $w\in T(c)$ of length $\ell$ has $\ell$ crossings, Theorem \ref{thm:alternating} below describes a reduced alternating diagram of the same knot with $c$ crossings, where $c$ is as given in the definition. Hence the crossing number of the knot is indeed $c$, and thus we use the notation $T(c)$ to describe the set. 

The set $T(c)$ is called partially double counted because every 2-bridge knot is represented by either one or two words in $T(c)$. The \textit{reverse} $r(w)$ of a word $w=a_1a_2\dots a_\ell$ is the word $r(w)=a_\ell a_{\ell-1}\dots a_1$. The \textit{reverse mirror of $w$}, denoted $\bar r(w)$, is the reverse of $w$ but with all pluses replaced by minuses and minuses replaced by pluses.

\begin{definition}
Define  the \textit{set of words of palindromic type} $T_p(c)$ to be the subset of $T(c)$ consisting of all words $w\in T(c)$ satisfying $w=r(w)$ when $c$ is odd and $w=\bar r(w)$ when $c$ is even.
\end{definition}

The following theorem explains how $T(c)$ and $T_p(c)$ can be used to enumerate the elements of $\mathcal{K}_c$. It is based on Schubert's classification of 2-bridge knots \cite{Sch}. This particular statement comes from Lemma 2.1 and Assumption 2.2 in \cite{CohenBound}.
\begin{theorem}
\label{thm:list}
Every $2$-bridge knot is represented by one or two words in $T(c)$. If a $2$-bridge knot is represented by a word $w$ of palindromic type, then $w$ is the unique word in $T(c)$ representing the knot. On the other hand, if a $2$-bridge knot is represented by a word $w$ that is not of palindromic type, then the knot is represented by exactly two words in $T(c)$: the words $w$ and $r(w)$ when $c$ is odd and the words $w$ and $\bar{r}(w)$ when $c$ is even.
\end{theorem}

While a word corresponds to a billiard table diagram with $\ell$ crossings, Cohen \cite{CohenBound} gave an algorithm for converting a word into a reduced alternating diagram with $c$ crossings. The crossings in the alternating diagram are arranged into two horizontal rows and are read from left to right. Crossings on the first row always appear in the format $\tikz[baseline=.6ex, scale = .4]{
\draw (0,0) -- (1,1);
\draw (0,1) -- (.3,.7);
\draw (.7,.3) -- (1,0);
}
~$, and crossings in the second row always appear in the format  $\tikz[baseline=.6ex, scale = .4]{
\draw (0,0) -- (.3,.3);
\draw (.7,.7) -- (1,1);
\draw (0,1) -- (1,0);
}
~$. See Figure \ref{fig:alternating} for an example. 
\begin{theorem}
\label{thm:alternating}
Each word $w$ in $T(c)$ corresponds to an alternating diagram where the runs $+$ and $--$ are replaced with crossings $\tikz[baseline=.6ex, scale = .4]{
\draw (0,0) -- (1,1);
\draw (0,1) -- (.3,.7);
\draw (.7,.3) -- (1,0);
}
~$ in the first row and the runs $-$ and $++$ are replaced with crossings $\tikz[baseline=.6ex, scale = .4]{
\draw (0,0) -- (.3,.3);
\draw (.7,.7) -- (1,1);
\draw (0,1) -- (1,0);
}
~$ in the second row.
\end{theorem}

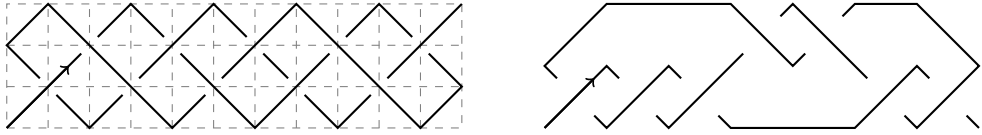
\begin{figure}[!h]
\centering

\begin{tikzpicture}[scale=.55]

	\draw[dashed, white!50!black] (0,0) rectangle (11,3);
	\foreach \x in {1,...,10}
		{\draw[dashed, white!50!black] (\x,0) -- (\x,3);}
	\foreach \x in {1,2}
		{\draw[dashed, white!50!black] (0,\x) -- (11, \x);}
	
	\draw[thick] (0,0) -- (1.8,1.8);
	\draw[thick] (2.2,2.2) -- (3,3) -- (3.8,2.2);
	\draw[thick] (4.2,1.8) -- (6,0) -- (7.8,1.8);
	\draw[thick] (8.2,2.2) -- (9,3) -- (9.8,2.2);
	\draw[thick] (10.2,1.8) -- (11,1) -- (10,0) -- (7,3) -- (5.2,1.2);
	\draw[thick] (4.8,.8) -- (4,0) -- (1,3) -- (0,2) -- (0.8,1.2);
	\draw[thick] (1.2,0.8) -- (2,0) -- (2.8,0.8);
	\draw[thick] (3.2,1.2) -- (5,3) -- (5.8,2.2);
	\draw[thick] (6.2,1.8) -- (6.8,1.2);
	\draw[thick] (7.2,0.8) -- (8,0) -- (8.8,0.8);
	\draw[thick] (9.2,1.2) -- (11,3);
	\draw[thick, ->] (0,0) -- (1.5,1.5);

\begin{scope}[xshift = 13 cm, scale=1.5]	

\draw[thick] (0,0) -- (1,1) -- (1.2,.8);
\draw[thick] (1.8,.2) -- (2,0) -- (3.2,1.2);
\draw[thick] (3.8,1.8) -- (4,2) -- (5.2,.8);
\draw[thick] (5.8,.2) -- (6,0) -- (7,1) -- (6,2) -- (5,2) -- (4.8,1.8);
\draw[thick] (4.2,1.2) -- (4,1) -- (3,2) -- (1,2) -- (0,1) -- (.2,.8);
\draw[thick] (.8,.2) -- (1,0) -- (2,1) -- (2.2,.8);
\draw[thick] (2.8,.2) -- (3,0) -- (5,0) -- (6,1) -- (6.2,.8);
\draw[thick] (6.8,.2) -- (7,0);

\draw[thick,->] (0,0) -- (.8,.8);

\end{scope}
	\end{tikzpicture}

\caption{A billiard table diagram corresponding to the word $+--+-++--+$ and its associated alternating diagram.}
    \label{fig:alternating}
\end{figure}

Murasugi \cite{Mur:genus} and Crowell \cite{Cro:genus} proved that the genus of an alternating knot $K$ is the genus of the Seifert surface coming from applying Seifert's algorithm to an alternating diagram of $K$. Our computation of the number $\bar{t}(c,g)$ of $2$-bridge knots with crossing number $c$ and genus $g$ depends on applying Seifert's algorithm to the alternating diagrams obtained from Theorem \ref{thm:alternating}.

\begin{theorem}
Let $K$ be an alternating knot with alternating diagram $D$ such that $D$ has $c$ crossings and performing Seifert's algorithm to $D$ results in $s$ Seifert circles. The genus of $K$ is 
\[g(K) = \frac{1}{2}\left(1 + c - s\right).\]
\end{theorem}

Ernst and Sumners \cite[Theorem 5]{ErnSum} computed the number $|\mathcal{K}_c|$ of 2-bridge knots. 

\begin{theorem}
\label{thm:ernstsumners}
The number $|\mathcal{K}_c|$ of 2-bridge knots with $c$ crossings where a knot and its mirror image are not counted separately is given by
\[
|\mathcal{K}_c| = 
\begin{cases}
\frac{1}{3}(2^{c-3}+2^{\frac{c-4}{2}}) & \text{ for }4 \geq c\equiv 0 \text{ mod }4,\\
\frac{1}{3}(2^{c-3}+2^{\frac{c-3}{2}}) & \text{ for }5\geq c\equiv 1 \text{ mod }4, \\
\frac{1}{3}(2^{c-3}+2^{\frac{c-4}{2}}-1) & \text{ for }6 \geq c\equiv 2 \text{ mod }4, \text{ and}\\
\frac{1}{3}(2^{c-3}+2^{\frac{c-3}{2}}+1) & \text{ for }3\geq c\equiv 3 \text{ mod }4.
\end{cases}
\]
\end{theorem}

\section{The number of 2-bridge knots with $c$ crossings of genus $g$}
\label{sec:recursions}

Define $t(c,g)$ and $t_p(c,g)$ to be the number of words in $T(c)$ and $T_p(c)$ respectively corresponding to knots with crossing number $c$ and genus $g$. In this section, we find recursive formulas for both $t(c,g)$ and $t_p(c,g)$, as well as express both counts as sums. Finally, we use Theorem \ref{thm:list} to prove Theorem \ref{thm:tbarformula}.



\subsection{Formulas for $T(c)$.}
 Following Cohen and Lowrance \cite{CoLow}, the set $T(c)$ can be partitioned into four subsets. From the parity of the crossing number $c$, we know the final run: it will be $+$ if $c$ is odd and $-$ if $c$ is even. We assume that $c$ is odd, but the case where $c$ is even is identical once the symbols $+$ and $-$ are interchanged. Since $c$ is odd, the penultimate run must be a single $-$ or a double $--$, and the anti-penultimate run must be a single $+$ or a double $++$. There are four cases for the final three runs of a word in $T(c)$: 
\begin{enumerate}
    \item $+-+$,
    \item $++--+$,
    \item $+--+$, and
    \item $++-+$.
\end{enumerate}
In order to develop a recursive formula for $t(c,g)$, we use the following replacements of the two penultimate runs in $w\in T(c)$ to obtain a word $w'$ in $T(c-1)$ or $T(c-2)$.
\begin{enumerate}
    \item The final runs $+-+$ can be replaced by $++-$, resulting in a word in $T(c-1)$.
    \item The final runs $++--+$ can be replaced by $+-$, resulting in a word in $T(c-1)$.
    \item The final runs $+--+$ can be replaced by $+$, resulting in a word in $T(c-2)$.
    \item The final runs $++-+$ can be replaced by $+$, resulting in a word in $T(c-2)$.
\end{enumerate}
 Observe that these replacements do not affect the modulo 3 condition on word length in Definition \ref{def:TC}.

\begin{example}
Table \ref{tab:Tcases} shows four words in $T(7)$, one corresponding to each of the cases above, along with the replacement word in $T(6)$ or $T(5)$. The subword being replaced and the corresponding replacement word are in parentheses.
\begin{center}
\begin{table}[h]
\begin{tabular}{|c| l| l|}
\hline
Case & Word in $T(7)$ & Replacement word \\
\hline
\hline
1 & $+ --++--(+-+)$ & $+ --++--(++-)$\\
\hline
2 & $+-+--(++--+)$ & $+-+--(+-)$\\
\hline
3 & $+--++-(+--+)$ & $+--++-(+)$\\
\hline
4 & $+--++-(++-+)$ & $+--++-(+)$\\
\hline
\end{tabular}
\caption{Four of the eleven words in $T(7)$ along with their replacement words in $T(6)$ and $T(5)$.}
\label{tab:Tcases}
\end{table}
\end{center}
\end{example}

Let $t(c)$ be the number of words in $T(c)$. Cohen and Lowrance \cite{CoLow} proved that $t(c)$ satisfies the Jacobsthal recurrence relation, that is, $t(c) = t(c-1) + 2t(c-2)$ and also has formula
\begin{equation}
    \label{eq:tc}
    t(c) = \frac{2^{c-2}-(-1)^c}{3}.
\end{equation}
The following lemma is a refinement of the recursive relation for $t(c)$ that takes into account different genera.

\begin{lemma}
\label{lemma:trecur}
Let $t(c,g)$ be the number of words in $T(c)$ corresponding to a knot with genus $g$. Then $t(c,g)$ satisfies the recurrence relation \[t(c,g) = t(c-1,g) + t(c-2,g-1) + t(c-2,g) + t(c-3,g-1) - t(c-3,g).\]

\begin{proof}[Proof of Lemma \ref{lemma:trecur}]
We split this into cases that correspond to the partition of $T(c)$ into four sets, except that case 2 is split into further subcases by considering the possibilities for the $(c-3)$rd run. Figure \ref{fig:t(c,g)} summarizes the cases. Define $t_1(c,g)$ to be the number of words in $T(c)$ with genus $g$ in case 1, that is, ending in the subword $+-+$. Similarly define $t_{2A}(c,g)$, $t_{2B}(c,g)$, $t_3(c,g)$, and $t_4(c,g)$. Since every possibility for the final runs is contained in exactly one of the cases, it follows that
\[t(c,g) = t_1(c,g)+t_{2A}(c,g) + t_{2B}(c,g) + t_3(c,g) + t_4(c,g).\]

In case 1, the final subword $+-+$ is replaced with $++-$, resulting in an arbitrary word in case 2 or 3 for crossing number $c-1$. Since the crossing number and the number of Seifert circles decrease by one, the genus remains the same after replacement. Thus $t_1(c,g) = t_{2A}(c-1,g) + t_{2B}(c-1,g) + t_3(c-1,g)$.  

In case 2A, the final subword $-++--+$ is replaced with $-+-$, resulting in an arbitrary word in case 1 for $c-1$ crossings. Since the crossing number and the number of Seifert circles decrease by one, the genus remains the same after replacement. Thus $t_{2A}(c,g) = t_{1}(c-1,g)$. 

In case 2B, the final subword $--++--+$ is replaced with $--+-$, resulting in an arbitrary word in case 4 for $c-1$ crossings. Since the crossing number decreases by one and the number of Seifert circles increases by one, the genus decreases by one. Thus $t_{2B}(c,g) = t_4(c-1,g-1)$.

In case 3, the final subword $+--+$ is replaced with $+$, resulting in an arbitrary word for crossing number $c-2$. Since the crossing number decreases by 2 and the number of Seifert circles remains the same, the genus decreases by one. Thus $t_{3}(c,g) = t(c-2,g-1).$ 

In case 4, the final subword $++-+$ is replaced with $+$, resulting in an arbitrary word for crossing number $c-2$. Since the crossing number and the number of Seifert circles decrease by $2$, the genus remains the same.  Thus $t_4(c,g) = t(c-2,g)$.  

Combining these cases, we find 
\begin{align*}
t(c,g) = & \; t_1(c,g) + t_{2A}(c,g) + t_{2B}(c,g) + t_{3}(c,g) + t_{4}(c,g)\\
= & \; t_{2A}(c-1,g) + t_{2B}(c-1,g) + t_3(c-1,g) + t_1(c-1,g)\\
& \; + t_4(c-1,g-1) + t(c-2,g-1) + t(c-2,g)\\
= & \; t_1(c-1,g) + t_{2A}(c-1,g) + t_{2B}(c-1,g) + t_3(c-1,g) + t_4(c-1,g) \\
& \; - t_4(c-1,g) + t_4(c-1,g-1) + t(c-2,g-1) + t(c-2,g)\\
= & \; t(c-1,g) + t(c-2,g-1) + t(c-2,g) + t(c-3,g-1) - t(c-3,g),
\end{align*}
our desired recursion.
\end{proof}
\end{lemma}

\begin{figure}[!h]
\[\begin{tikzpicture}


\draw (-.25,1.5) node{\small{Terminal}};
\draw (-.25,1.1) node{\small{String}};

\draw (1.9,1.5) node{\small{Alternating}};
\draw (1.9,1.1) node{\small{Diagram}};

\draw (4,1.5) node{\small{Seifert}};
\draw (4,1.1) node{\small{State}};

\draw(-1,-.75) rectangle (5.5,.75);
\draw(0,0) node{${\scriptstyle +-+}$};
\begin{scope}[scale=.5, rounded corners = 1mm, xshift = 2.5cm, yshift = -1cm]
\draw (0,0) -- (1.3, 1.3);
\draw (0,1) -- (.3,.7);
\draw (.7,.3) -- (1,0) -- (2,0) -- (3,1) -- (2,2) -- (1.7,1.7);
\draw (0,2) -- (1,2) -- (2.3,.7);
\draw (2.7,.3) -- (3,0);

\draw[->] (.5, .5) -- (.1,.1);
\draw[->] (.7,.3) -- (.9,.1);
\draw[->] (2.5, .5) -- (2.9,.9);
\draw[->] (2.7,.3) -- (2.9,.1);
\draw[->] (1.5, 1.5) -- (1.9,1.1);
\draw[->] (1.3,1.3) -- (1.1,1.1);

\end{scope}

\begin{scope}[scale=.5, rounded corners = 1mm, xshift = 6.75cm, yshift = -1cm]
\draw[->] (0,1) -- (.4,.5) -- (0,0);
\draw[->] (0,2) -- (1,2) -- (1.4,1.5) -- (.6,.5) -- (1,0) -- (2,0) -- (2.5,.4) -- (3,0);
\draw[->] (2,1) -- (2.5,.6) -- (3,1) -- (2,2) -- (1.6,1.5) -- (2,1);
\end{scope}

\draw[->] (5.5,0) -- (7,0);
\draw (6.25,.5) node{\small{Case 1}};

\begin{scope}[xshift = 8 cm]
\draw (-.25,1.5) node{\small{Terminal}};
\draw (-.25,1.1) node{\small{String}};

\draw (1.9,1.5) node{\small{Alternating}};
\draw (1.9,1.1) node{\small{Diagram}};

\draw (4,1.5) node{\small{Seifert}};
\draw (4,1.1) node{\small{State}};
\draw(-1,-.75) rectangle (5.5,.75);
\draw(0,0) node{${\scriptstyle ++-}$};
\begin{scope}[scale=.5, rounded corners = 1mm, xshift = 3.5cm, yshift = -1cm]
\draw (0,0) -- (1,0) -- (2,1) -- (1.7,1.3);
\draw (1.3,1.7) -- (1,2) -- (0,1);
\draw (0,2) -- (0.3,1.7);
\draw (.7,1.3) -- (1,1) -- (2,2);
\draw[->] (0.5,1.5) -- (.9,1.9);
\draw[->] (.7,1.3) -- (.9,1.1);
\draw[->] (1.5,1.5) -- (1.9,1.9);
\draw[->] (1.7, 1.3) -- (1.9,1.1);

\end{scope}

\begin{scope}[scale=.5, rounded corners = 1mm, xshift = 7.5cm, yshift = -1cm]
\draw[->] (0,2) -- (.5,1.6) -- (1,2) -- (1.5,1.6) -- (2,2);
\draw[->] (0,1) -- (.5, 1.4) -- (1,1) -- (1.5,1.4) -- (2,1) -- (1,0) -- (0,0);
\end{scope}
\end{scope}

\begin{scope}[yshift = -2cm]

\draw(-1,-.75) rectangle (5.5,.75);
\draw(0,0) node{${\scriptstyle -++--+}$};
\begin{scope}[scale=.5, rounded corners = 1mm, xshift = 3cm, yshift = -1cm]
\draw (-1,0) -- (1,0) -- (2,1) -- (2.3,.7);
\draw (2.7,.3) -- (3,0);
\draw (-1,2) -- (0,1) -- (.3,1.3);
\draw (-.3,1.7) -- (0,2) -- (1.3,.7);
\draw (-1,1) -- (-.7,1.3);
\draw (1.7,.3) -- (2,0) -- (3,1) -- (2,2) -- (1,2) -- (.7,1.7);
\draw[->] (.3,1.3) -- (.1,1.1);
\draw[->] (.5,1.5) -- (.9,1.1);
\draw[->] (1.5,.5) -- (1.9,.9);
\draw[->] (1.7,.3) -- (1.9,.1);
\draw[->] (2.5,.5) -- (2.9,.9);
\draw[->] (2.7,.3) -- (2.9,.1);
\draw[->] (-.5,1.5) -- (-.9,1.9);
\draw[->] (-.3,1.7) -- (-.1,1.9);

\end{scope}

\begin{scope}[scale=.5, rounded corners = 1mm, xshift = 7.5cm, yshift = -1cm]
 \draw[->] (0,2) arc (90:-270:.4cm and .5cm);
 \draw[->] (-1,0) -- (1,0) -- (1.5,.4) -- (2,0) -- (2.5,.4) -- (3,0);
 \draw[->] (1.5,2) -- (1,2) -- (.6,1.5) -- (1,1) -- (1.5,.6) -- (2,1) -- (2.5,.6) -- (3,1) -- (2,2) -- (1.5,2);
 \draw[->] (-1,1) -- (-.6,1.5) -- (-1,2);
\end{scope}

\draw[->] (5.5,0) -- (7,0);
\draw (6.25,.5) node{\small{Case 2A}};

\end{scope}

\begin{scope}[xshift = 8cm, yshift=-2cm]

\draw(-1,-.75) rectangle (5.5,.75);

\draw(0,0) node{${\scriptstyle -+-}$};
\begin{scope}[scale=.5, rounded corners = 1mm, xshift = 3.5cm, yshift = -1cm]
\draw (-1,0) -- (0,0) -- (1.3,1.3);
\draw (1.7,1.7)--(2,2);
\draw (-1,2) -- (0.3,0.7);
\draw (0.7,0.3) -- (1,0) -- (2,1) -- (1,2) -- (0,2) -- (-.3,1.7);
\draw (-1,1) -- (-.7,1.3);

\draw[->] (-.3,1.7) -- (-.1,1.9);
\draw[->] (-.5,1.5) -- (-.9,1.9);
\draw[->] (0.5,0.5) -- (0.9, 0.9);
\draw[->] (0.3,0.7) -- (0.1,0.9);
\draw[->] (1.5,1.5) -- (1.9,1.1);
\draw[->] (1.7, 1.7) -- (1.9, 1.9);

\end{scope}

\begin{scope}[scale=.5, rounded corners = 1mm, xshift = 7.5cm, yshift = -1cm]
\draw[->] (-1,0) --(0,0) -- (.4,.5) -- (0,1) -- (-.4,1.5) -- (0,2)-- (1,2) --(1.5,1.6) -- (2,2);
\draw[->] (1,1) -- (1.5,1.4) -- (2,1) -- (1,0) -- (0.6,0.5) -- (1,1);
\draw[->] (-1,1) -- (-.6,1.5) -- (-1,2);
\end{scope}
\end{scope}

\begin{scope}[yshift = -4cm]

\draw(-1,-.75) rectangle (5.5,.75);
\draw(0,0) node{${\scriptstyle --++--+}$};
\begin{scope}[scale=.5, rounded corners = 1mm, xshift = 3cm, yshift = -1cm]
\draw (-.3,.3) -- (0,0) -- (1,0) -- (2,1) -- (2.3,.7);
\draw (2.7,.3) -- (3,0);
\draw (-1,0) -- (.3,1.3);
\draw (-1,2) -- (0,2) -- (1.3,.7);
\draw (1.7,.3) -- (2,0) -- (3,1) -- (2,2) -- (1,2) -- (.7,1.7);
\draw (-1,1) -- (-.7,.7);
\draw[->] (.3,1.3) -- (.1,1.1);
\draw[->] (.5,1.5) -- (.9,1.1);
\draw[->] (1.5,.5) -- (1.9,.9);
\draw[->] (1.7,.3) -- (1.9,.1);
\draw[->] (2.5,.5) -- (2.9,.9);
\draw[->] (2.7,.3) -- (2.9,.1);
\draw[->] (-.5,.5) -- (-1,0);
\draw[->] (-.3,.3) -- (-.1,.1);

\end{scope}

\begin{scope}[scale=.5, rounded corners = 1mm, xshift = 7.5cm, yshift = -1cm]
 \draw[->] (-1,2) -- (0,2) -- (.4,1.5) -- (0,1) -- (-.4,.5) -- (0,0) -- (1,0) -- (1.5,.4) -- (2,0) -- (2.5,.4) -- (3,0);
 \draw[->] (1.5,2) -- (1,2) -- (.6,1.5) -- (1,1) -- (1.5,.6) -- (2,1) -- (2.5,.6) -- (3,1) -- (2,2) -- (1.5,2);
 \draw[->] (-1,1) -- (-.6,.5) -- (-1,0);
\end{scope}

\draw[->] (5.5,0) -- (7,0);
\draw (6.25,.5) node{\small{Case 2B}};

\end{scope}

\begin{scope}[xshift = 8cm, yshift = -4cm]

\draw(-1,-.75) rectangle (5.5,.75);

\draw(0,0) node{${\scriptstyle --+-}$};
\begin{scope}[scale=.5, rounded corners = 1mm, xshift = 3.5cm, yshift = -1cm]
\draw (-.3,.3) -- (0,0) -- (1.3,1.3);
\draw (1.7,1.7)--(2,2);
\draw (-1,0) -- (0,1) -- (0.3,0.7);
\draw (-1,1) -- (-.7,.7);
\draw (0.7,0.3) -- (1,0) -- (2,1) -- (1,2) -- (0,2) -- (-1,2);

\draw[->] (0.5,0.5) -- (0.9, 0.9);
\draw[->] (0.3,0.7) -- (0.1,0.9);
\draw[->] (1.5,1.5) -- (1.9,1.1);
\draw[->] (1.7, 1.7) -- (1.9, 1.9);
\draw[->] (-.5,.5) -- (-.9,.1);
\draw[->] (-.3,.3) -- (-.1,.1);

\end{scope}

\begin{scope}[scale=.5, rounded corners = 1mm, xshift = 7.5cm, yshift = -1cm]
\draw[->] (0,1) arc (90:450:.4cm and .5cm);
\draw[->] (-1,1) -- (-.6,.5) -- (-1,0);
\draw[->] (-1,2) -- (1,2) --(1.5,1.6) -- (2,2);
\draw[->] (1,1) -- (1.5,1.4) -- (2,1) -- (1,0) -- (0.6,0.5) -- (1,1);
\end{scope}
\end{scope}

\begin{scope}[yshift = -6cm]

\draw(-1,-.75) rectangle (5.5,.75);
\draw(0,0) node{${\scriptstyle +--+}$};
\begin{scope}[scale=.5, rounded corners = 1mm, xshift = 2.5cm, yshift = -1cm]
\draw[->] (0,0) -- (.5,.4) -- (1,0) -- (1.5,.4) -- (2,0) -- (2.5,.4) -- (3,0);
\draw[->] (0,1) -- (.5,.6) -- (1,1) -- (1.5,.6) -- (2,1) -- (2.5,.6) -- (3,1) -- (2,2) -- (0,2);

\end{scope}

\begin{scope}[scale=.5, rounded corners = 1mm, xshift = 7.5cm, yshift = -1cm]
 \draw[->] (-1,2) -- (0,2) -- (.4,1.5) -- (0,1) -- (-.4,.5) -- (0,0) -- (1,0) -- (1.5,.4) -- (2,0) -- (2.5,.4) -- (3,0);
 \draw[->] (1.5,2) -- (1,2) -- (.6,1.5) -- (1,1) -- (1.5,.6) -- (2,1) -- (2.5,.6) -- (3,1) -- (2,2) -- (1.5,2);
 \draw[->] (-1,1) -- (-.6,.5) -- (-1,0);
\end{scope}

\draw[->] (5.5,0) -- (7,0);
\draw (6.25,.5) node{\small{Case 3}};

\end{scope}

\begin{scope}[xshift = 8cm, yshift = -6cm]

\draw(-1,-.75) rectangle (5.5,.75);

\draw(0,0) node{${\scriptstyle +}$};
\begin{scope}[scale=.5, rounded corners = 1mm, xshift = 3.5cm, yshift = -1cm]
\draw (0,1) -- (.3,.7);
\draw (.7,.3) -- (1,0);
\draw (0,0) -- (1,1) -- (0,2);

\draw[->] (.5,.5) -- (.9,.9);
\draw[->] (.7,.3) -- (.9,.1);

\end{scope}

\begin{scope}[scale=.5, rounded corners = 1mm, xshift = 7.5cm, yshift = -1cm]
\draw[->] (0,0) -- (.5,.4) -- (1,0);
\draw[->] (0,1) -- (.5,.6) -- (1,1) -- (0,2);
\end{scope}
\end{scope}

\begin{scope}[yshift = -8cm]

\draw(-1,-.75) rectangle (5.5,.75);
\draw(0,0) node{${\scriptstyle ++-+}$};
\begin{scope}[scale=.5, rounded corners = 1mm, xshift = 2.5cm, yshift = -1cm]
\draw (0,0) -- (2,0) -- (3,1) -- (2,2) -- (1.7,1.7);
\draw (1.3,1.3) -- (1,1) -- (0,2);
\draw (0,1) -- (.3,1.3);
\draw (.7,1.7) -- (1,2) -- (2.3,.7);
\draw (2.7,0.3) -- (3,0);

\draw[->] (0.5, 1.5) -- (0.1, 1.9);
\draw[->] (0.7,1.7) -- (0.9,1.9);
\draw[->] (1.5, 1.5) -- (1.9,1.1);
\draw[->] (1.3,1.3) -- (1.1, 1.1);
\draw[->] (2.5,0.5) -- (2.9,0.9);
\draw[->] (2.7,0.3) -- (2.9, 0.1);
\end{scope}

\begin{scope}[scale=.5, rounded corners = 1mm, xshift = 7.5cm, yshift = -1cm]
\draw[->] (0,0) -- (2,0) -- (2.5,.4) -- (3,0);
\draw[->] (0,1) -- (.4,1.5) -- (0,2);
\draw[->] (1,2) arc (90:-270:.4 cm and .5cm);
\draw[->] (2,1) -- (2.5,.6) -- (3,1) -- (2,2) -- (1.6,1.5) -- (2,1);
\end{scope}
\draw[->] (5.5,0) -- (7,0);
\draw (6.25,.5) node{\small{Case 4}};

\end{scope}

\begin{scope}[xshift = 8cm, yshift = -8cm]
\draw(-1,-.75) rectangle (5.5,.75);

\draw(0,0) node{${\scriptstyle +}$};
\begin{scope}[scale=.5, rounded corners = 1mm, xshift = 3.5cm, yshift = -1cm]
\draw (0,1) -- (.3,.7);
\draw (.7,.3) -- (1,0);
\draw (0,0) -- (1,1) -- (0,2);

\draw[->] (.5,.5) -- (.9,.9);
\draw[->] (.7,.3) -- (.9,.1);

\end{scope}

\begin{scope}[scale=.5, rounded corners = 1mm, xshift = 7.5cm, yshift = -1cm]
\draw[->] (0,0) -- (.5,.4) -- (1,0);
\draw[->] (0,1) -- (.5,.6) -- (1,1) -- (0,2);
\end{scope}
\end{scope}

\end{tikzpicture}\]

\caption{Terminal strings, alternating diagrams, and partial Seifert states corresponding to the cases in Lemma \ref{lemma:trecur}.}
\label{fig:t(c,g)}
\end{figure}
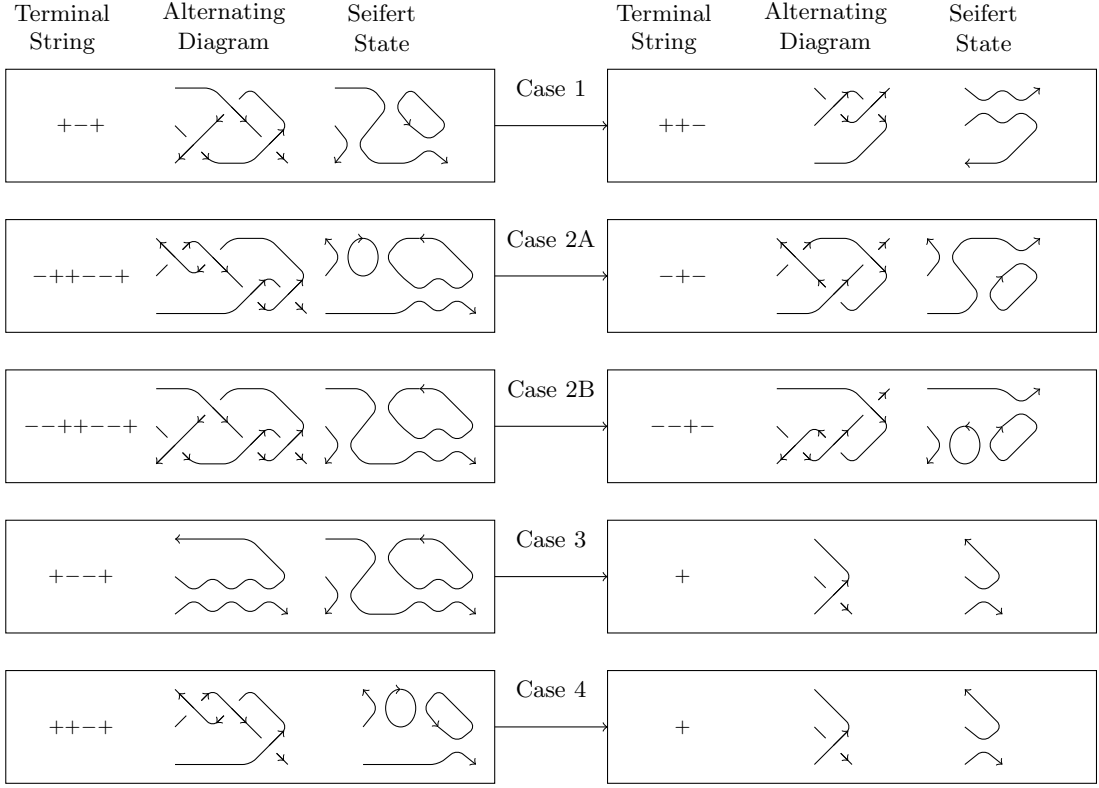

We use the recursive formula in Lemma \ref{lemma:trecur} to express $t(c,g)$ as a sum.

\begin{lemma}
\label{lemma:tcg}
The number of words $t(c,g)$ in $T(c)$ corresponding to knots of genus $g$ can be expressed as  \[t(c,g)=(-1)^{c-1}\sum_{n=0}^{c-2g-1}(-1)^n \binom{n+2g-1}{n}.\]

\begin{proof}

We proceed by induction on $c$. The genus of a 2-bridge knot is at least 1. The base cases $t(3,1)=1$, $t(4,1)=1$, $t(5,1)=1$, and $t(5,2)=1$ can be obtained by direct computation. When $3\leq c \leq 5$ and $g$ is a value not mentioned in the previous sentence, then summation's upper limit is negative, and hence $t(c,g)=0$. Thus, the result holds for crossing numbers 3, 4, and 5. 

\indent For the inductive step, we have

\begin{align*}
    t(c, g) = &\; t(c-2,g) + t(c-1,g) - t(c-3,g) + t(c-2,g-1) + t(c-3,g-1)\\
    = & \; (-1)^{c-3}\sum_{n=0}^{c-2g-3}(-1)^n{n+2g-1 \choose n}\\ 
    & \; + (-1)^{c-2}\sum_{n=0}^{c-2g-2}(-1)^n{n+2g-1 \choose n} - (-1)^{c-4}\sum_{n=0}^{c-2g-4}(-1)^n{n + 2g - 1 \choose n}\\
    & \; + (-1)^{c-3}\sum_{n=0}^{c-2g-1}(-1)^n{n+2g-3 \choose n} + (-1)^{c-4}\sum_{n=0}^{c-2g-2}(-1)^n{n+2g-3 \choose n}\\
    = & \; (-1)^{c-3}\sum_{n=0}^{c-2g-3}(-1)^n{n+2g-1 \choose n} - {c-4 \choose c-2g-3} + {c-3\choose c-2g-2} + {c-4 \choose c-2g-1}\\
    = & \; (-1)^{c-1} \sum_{n=0}^{c-2g-3}(-1)^n{n+2g-1 \choose n} + {c-4 \choose c-2g-2} + {c-4 \choose c-2g-1}\\
    = & \; (-1)^{c-1} \sum_{n=0}^{c-2g-3}(-1)^n{n+2g-1 \choose n} + {c-3 \choose c-2g-1}\\
    = & \; (-1)^{c-1} \sum_{n=0}^{c-2g-3}(-1)^n{n+2g-1 \choose n} - {c-3 \choose c-2g-2} + {c-2 \choose c-2g-1}\\
    = & \; (-1)^{c-1}\sum_{n=0}^{c-2g-1}(-1)^n \binom{n+2g-1}{n},
\end{align*}
where the second equality uses the inductive hypothesis several times, the third equality follows from terms cancelling in the second and third sums and the fourth and fifth sums, and the remaining equalities follow from the relation ${n \choose k} = {n-1 \choose k-1} + {n-1 \choose k}$.
\end{proof}
\end{lemma}

\subsection{Formulas for $T_p(c)$}
Recall that $t_p(c,g)$ is the number of words of palindromic type in $T_p(c)$ corresponding to a knot of genus $g$. We use a similar algorithm to find a recursive and summation formula for $t_p(c,g)$. The arguments when $c$ is odd or even are similar with slightly different cases, and so we only give details when $c$ is even. Cohen and Lowrance \cite{CoLow} showed that if $c=2k$, then the $(k-1)$st and the $k$th runs must be one of the following four cases:
\begin{enumerate}
    \item [$(1_{p})$] a single $+$ followed by a single $-$,
    \item [$(2_{p})$] a double $++$ followed by a double $--$,
    \item [$(3_{p})$] a single $+$ followed by a double $--$,
    \item [$(4_{p})$] a double $++$ followed by a single $-$.
\end{enumerate}
As was the case with $T(c)$, the runs in the four cases above can be replaced as seen in Figure \ref{fig:tp} to form words in $T_p(c-2)$ and $T_p(c-4)$. 

Let $t_p(c)$ be the number of words in $T_p(c)$.
Cohen and Lowrance \cite{CoLow} proved that 
\begin{equation}
    \label{eq:tpc}
    t_p(c) = \frac{2^{\left\lfloor\frac{c-1}{2}\right\rfloor} - (-1)^{\left\lfloor\frac{c-1}{2}\right\rfloor}}{3},
\end{equation}
and that $t_p(c)$ satisfies the recurrence relation $t_p(c) = t_p(c-2)+2t_p(c-4)$. The following lemma is a refinement of the recurrence relation for $t_p(c)$ that takes into account different genera.

\begin{lemma}
\label{lemma:tprecur}
Let $t_p(c,g)$ be the number of words of palindromic type, crossing number $c$ and genus $g$. Then $t_p(c,g)$ satisfies the recurrence relation $t_p(c,g) = t_p(c-2,g-1) + t_p(c-4,g) + t_p(c-4,g-1)$. 

\begin{proof} Suppose $c=2k$ is even; the argument when $c$ is odd is similar and left to the reader. A subword of a word $w$ in $T_p(c)$ is a \textit{center string} if it contains the $(k+1-j)$th through the $(k+j)$th runs of $w$ for some $j\in\mathbb{N}$. Similar to the process for $t(c,g)$, we break $t_p(c,g)$ into cases 1, 2, 3A, 3B, 4A, and 4B, matching the cases above but where we consider the possibilities for the $(k-2)$nd run in case 3 and 4. Figure \ref{fig:tp} summarizes the cases. Define $t_{p1}(c,g)$ to be the number of words of palindromic type with crossing number $c$ and genus $g$ and center string $+-+-$. 
We similarly define $t_{p2}(c,g)$, $t_{p3A}(c,g)$, $t_{p3B}(c,g)$, $t_{p4A}(c,g)$, and $t_{p4B}(c,g)$. Since the cases form a partition of $t_p(c,g)$, it follows that $t_p(c,g) = t_{p1}(c,g) + t_{p2}(c,g) + t_{p3A}(c,g) + t_{p3B}(c,g) + t_{p4A}(c,g) + t_{p4B}(c,g)$. 




In case $1$, the center string $+-+-$ is replaced by $++--$, resulting in an arbitrary word in case 2 or 3 for $T_p(c-2)$. Since the crossing number decreases by 2 and the number of Seifert circles remains the same, the genus decreases by one after replacement. Thus $t_{p1}(c,g) = t_{p2}(c-2,g-1) + t_{p3A}(c-2,g-1) + t_{p3B}(c-2,g-1)$.

In case $2$, the center string $++--++--$ is replaced by $+-$, resulting in an arbitrary word in case 1, 4A, or 4B for $T_p(c-2)$. Since the crossing number decreases by 2 and the number of Seifert circles remains the same, the genus decreases by one after replacement. Thus $t_{p2}(c,g) = t_{p1}(c-2,g-1) + t_{p4A}(c-2,g-1) + t_{p4B}(c-2,g-1)$.

In case 3A, the center string $-+--++-+$ is replaced by $-+$, resulting in an arbitrary word in case 1, 4A, or 4B for $T_p(c-4)$. Since the number of crossings decreases by 4 and the number of Seifert circles decreases by 2, the genus decreases by one after replacement. Thus $t_{p3A}(c,g) = t_{p1}(c-4,g-1) + t_{p4A}(c-4,g-1) + t_{p4B}(c-4,g-1)$.

In case 3B, the center string $--+--++-++$ is replaced by $--++$, resulting in an arbitrary word in case 2, 3A, or 3B for $T_p(c-4)$. Since both the number of crossings and the number of Seifert circles decreases by 4, the genus remains the same after replacement. Thus $t_{p3A}(c,g) = t_{p2}(c-4,g) + t_{p3A}(c-4,g) + t_{p3B}(c-4,g)$.

In case 4A, the center string $-++-+--+$ is replaced by $-+$, resulting in an arbitrary word in case 1, 4A, or 4B for $T_p(c-4)$. Since both the number of crossings and the number of Seifert circles decreases by 4, the genus remains the same after replacement. Thus $t_{p4A}(c,g) = t_{p1}(c-4,g) + t_{p4A}(c-4,g) + t_{p4B}(c-4,g)$.

In case 4B, the center string $--++-+--++$ is replaced by $--++$, resulting in an arbitrary word in case 2, 3A, or 3B for $T_p(c-4)$. Since the number of crossings decreases by 4 and the number of Seifert circles decreases by 2, the genus decreases by one after replacement. Thus $t_{p4B}(c,g) = t_{p2}(c-4,g-1) + t_{p3A}(c-4,g-1) + t_{p3B}(c-4,g-1)$.

Putting these recursive rules together yields
\begin{align*}
t_p(c,g) = & \;  t_{p2}(c-2,g-1) + t_{p3A}(c-2,g-1) + t_{p3B}(c-2,g-1) + t_{p1}(c-2,g-1) + t_{p4A}(c-2,g-1) \\
& \;+ t_{p4B}(c-2,g-1) +  t_{p1}(c-4,g-1) + t_{p4A}(c-4,g-1)+ t_{p4B}(c-4,g-1) + t_{p2}(c-4,g) \\
& \; + t_{p3A}(c-4,g) + t_{p3B}(c-4,g) + t_{p1}(c-4,g) + t_{p4A}(c-4,g) + t_{p4B}(c-4,g)\\
& \; + t_{p2}(c-4,g-1) + t_{p3A}(c-4,g-1) + t_{p3B}(c-4,g-1)\\
 = & \; t_{p1}(c-2,g-1) + t_{p2}(c-2,g-1) + t_{p3A}(c-2,g-1) + t_{p3B}(c-2,g-1) + t_{p4A}(c-2,g-1)\\
 & \; + t_{p4B}(c-2,g-1) + t_{p1}(c-4,g-1) + t_{p2}(c-4,g-1) + t_{p3A}(c-4,g-1) + t_{p3B}(c-4,g-1)\\
 & \; + t_{p4A}(c-4,g-1) + t_{p4B}(c-4,g-1) + t_{p1}(c-4,g) + t_{p2}(c-4,g) + t_{p3A}(c-4,g)\\
 & \; + t_{p3B}(c-4,g) + t_{p4A}(c-4,g) + t_{p4B}(c-4,g)\\
= & \; t_p(c-2,g-1) + t_p(c-4,g-1) + t_p(c-4,g).
\end{align*}
\end{proof}
\end{lemma}

The recurrence relation for $t_p(c,g)$ only involves words corresponding to knots with 2 or 4 fewer crossings. Thus the recurrence relation respects the parity of the crossing. Since $t_p(3,g)=t_p(4,g)$ and $t_p(5,g)=t_p(6,g)$ for all $g$, we have the following corollary.

\begin{corollary}
\label{t_pevenandoddsame}
Let $c$ be odd. Then $t_p(c,g) = t_p(c+1,g)$ for all $g$.
\end{corollary}



We use the recursive formula in Lemma \ref{lemma:tprecur} to express $t_p(c,g)$ as a sum.

\begin{lemma}
\label{lemma:tp(c,g)}
The number of words $t_p(c,g)$ in $T_p(c)$ corresponding to knots of genus $g$ can be expressed as 
\[t_p(c,g)=(-1)^{c'-g-1}\sum_{n=0}^{c'-g-1}(-1)^n{n+g-1 \choose n},\] where $c' = \frac{c}{2}$ if $c$ is even and $c'=\frac{c+1}{2}$ if $c$ is odd.
\end{lemma}


\begin{proof}
Proceed by induction. The base cases $t_p(3,1)=1$, $t_p(4,1)=1$, $t_p(5,1)=0$, $t_p(5,2)=1$, $t_p(6,1)=0$, and $t_p(6,2)=1$ can be obtained by direct computation. When $3 \leq c \leq 6$ and $g$ is a value unmentioned in the previous sentence, the summation's upper limit is negative, and hence $t_p(c,g)=0$. Thus the result holds for  crossing numbers 3, 4, 5, and 6.\\
\indent For the inductive step, we have
\begin{align*}
    t_p(c,g) = & t_p(c-4,g)+t_p(c-2,g-1)+t_p(c-2,g-1)\\ 
    = & \; (-1)^{c'-g-3}\sum_{n=0}^{c'-g-3}(-1)^n{n+g-1 \choose n}\\
    & \; + (-1)^{c'-g-2}\sum_{n=0}^{c'-g-2}(-1)^n{n+g-2 \choose n} + (-1)^{c'-g-1}\sum_{n=0}^{c'-g-1}(-1)^n{n+g-2 \choose n}\\
    = & \; (-1)^{c'-g-3}\sum_{n=0}^{c'-g-3}(-1)^n{n+g-1 \choose n} + {c'-3 \choose c'-g-1}\\
    = & \; (-1)^{c'-g-3}\sum_{n=0}^{c'-g-3}(-1)^n{n+g-1 \choose n} + {c'-3 \choose c'-g-1} + {c'-3 \choose c'-g-2} - {c'-3 \choose c'-g-2}\\
    = & \; (-1)^{c'-g-3}\sum_{n=0}^{c'-g-3}(-1)^n{n+g-1 \choose n} + {c'-2 \choose c'-g-1} - {c'-3 \choose c'-g-2}\\
   = & \; (-1)^{c'-g-1}\sum_{n=0}^{c'-g-1}(-1)^n{n+g-1 \choose n},
\end{align*}
where the second equality uses the inductive hypothesis several times, the third equality follows from terms cancelling in the second and third summations, the fourth equality follows by adding zero, and the fifth equality follows from the relation ${n \choose k} = {n-1 \choose k-1} + {n-1 \choose k}$. 
\end{proof}

\begin{figure}[h]
\[\begin{tikzpicture}

\draw ( 1,1.5) node{\small{Alternating}};
\draw (1,1.1) node{\small{Diagram}};

\draw (5 ,1.5) node{\small{Seifert}};
\draw (5,1.1) node{\small{State}};

\draw(-1,-1.25) rectangle (7,.75);
\draw(3,-1) node{${\scriptstyle +-+-}$};
\begin{scope}[scale=.33, yshift =  -1cm, xshift = -2.5cm]
\draw[white] (-.2,-.7) rectangle (10.2,2.7);
\draw (1,-.5) rectangle (3,2.5);
\draw (2,1) node{$R$};
\draw (8,1) node[rotate = 180]{$\overline{R}$};
\draw (7,-.5) rectangle (9,2.5);
\draw (0,0) -- (1,0);
\draw (9,2) -- (10,2);
\begin{scope}[rounded corners = .5mm]
	\draw (3,0) -- (4.3,1.3);
	\draw (3,1) -- (3.3,.7);
	\draw (3,2) -- (4,2) -- (5.3,.7);
	\draw (4.7,1.7) -- (5,2) -- (6,2) -- (7,1);
	\draw (3.7,.3) -- (4,0) -- (5,0) -- (6.3,1.3);
	\draw (6.7,1.7) -- (7,2);
	\draw (5.7,.3) -- (6,0) -- (7,0);
\end{scope}
\draw[->] (3.5,.5) -- (3.9,.9);
\draw[->] (3.7,.3) -- (3.9,.1);
\draw[->] (4.5,1.5) -- (4.1,1.9);
\draw[->] (4.7,1.7) -- (4.9,1.9);
\draw[->] (5.5,.5) -- (5.9,.9);
\draw[->] (5.3,.7) -- (5.1,.9);
\draw[->] (6.5,1.5) -- (6.9,1.1);
\draw[->] (6.7,1.7) -- (6.9,1.9);

\end{scope}

\begin{scope}[scale=.33, yshift = -1cm, xshift = 10cm]
\draw (1,-.5) rectangle (3,2.5);
\draw (7,-.5) rectangle (9,2.5);
\draw (0,0) -- (1,0);
\draw (9,2) -- (10,2);
\begin{scope}[rounded corners = 1mm]
	\draw[->] (3,0) -- (3.5,.4) -- (4,0) -- (5,0) -- (5.4,.5) -- (5,1) -- (4.6,1.5) -- (5,2) -- (6,2) -- (6.5,1.6) -- (7,2);
	\draw[->] (3,2) -- (4,2) -- (4.4,1.5) -- (4,1) -- (3.5,.6) -- (3,1);
	\draw[->] (7,0) -- (6,0) -- (5.6,.5) -- (6,1) -- (6.5,1.4) -- (7,1);
\end{scope}
\draw[densely dashed] (1,0) -- (3,0);
\draw[densely dashed, rounded corners=1mm] (3,2) -- (2.6,1.5) -- (3,1);
\draw[densely dashed] (7,2) -- (9,2);
\draw[densely dashed, rounded corners=1mm] (7,1) -- (7.4,.5) -- (7,0);
\end{scope}

\draw[->] (7,-.25) -- (7.5,-.25);
\draw (7.25,.25) node{\small{1}};

\begin{scope}[xshift = 8.5 cm]

\draw (.5,1.5) node{\small{Alternating}};
\draw (.5,1.1) node{\small{Diagram}};

\draw (3.5,1.5) node{\small{Seifert}};
\draw (3.5,1.1) node{\small{State}};
\draw(-1,-1.25) rectangle (5,.75);
\draw(2,-1) node{${\scriptstyle ++--}$};
\begin{scope}[scale=.33, xshift = -2cm, yshift = -1cm]
\draw (1,-.5) rectangle (3,2.5);
\draw (2,1) node{$R$};
\draw (6,1) node[rotate = 180]{$\overline{R}$};
\draw (5,-.5) rectangle (7,2.5);
\draw (0,0) -- (1,0);
\draw (7,2) -- (8,2);
\begin{scope}[rounded corners = 1mm]
	\draw (3,0) -- (4,0) -- (5,1);
	\draw (3,1) -- (3.3,1.3);
	\draw (3.7,1.7) -- (4,2) -- (5,2);
	\draw (3,2) -- (4.3,.7);
	\draw (4.7,.3) -- (5,0);
\end{scope}

\draw[->] (3.5,1.5) -- (3.1,1.9);
\draw[->] (3.7,1.7) -- (3.9,1.9);
\draw[->] (4.5,.5) -- (4.9,.9);
\draw[->] (4.3,.7) -- (4.1,.9);

\end{scope}

\begin{scope}[scale=.33, xshift = 6.5cm, yshift = -1cm]
\draw (1,-.5) rectangle (3,2.5);
\draw (5,-.5) rectangle (7,2.5);
\draw (0,0) -- (1,0);
\draw (7,2) -- (8,2);
\begin{scope}[rounded corners = 1mm]
	\draw[->] (3,1) -- (3.4,1.5) -- (3,2);
	\draw[->] (5,0) -- (4.6,.5) -- (5,1);
	\draw[->] (3,0) -- (4,0) -- (4.4,.5) -- (4,1) -- (3.6,1.5) -- (4,2) -- (5,2);
\end{scope}
\draw[densely dashed] (1,0) -- (3,0);
\draw[densely dashed, rounded corners=1mm] (3,2) -- (2.6,1.5) -- (3,1);
\draw[densely dashed] (5,2) -- (7,2);
\draw[densely dashed, rounded corners=1mm] (5,1) -- (5.4,.5) -- (5,0);
\end{scope}
\end{scope}


\begin{scope}[yshift=-2.5cm]

\draw(-1,-1.25) rectangle (7,.75);
\draw(3,-1) node{${\scriptstyle ++--++--}$};
\begin{scope}[scale=.33, yshift =  -1cm, xshift = -2.25cm]
\draw (1,-.5) rectangle (3,2.5);
\draw (2,1) node{$R$};
\draw (8,1) node[rotate = 180]{$\overline{R}$};
\draw (7,-.5) rectangle (9,2.5);
\draw (0,0) -- (1,0);
\draw (9,2) -- (10,2);
\begin{scope}[rounded corners = 1mm]
	\draw (3,0) -- (4,0) -- (5.3,1.3);
	\draw (5.7,1.7) -- (6,2) --(7,2);
	\draw (3,1) -- (3.3,1.3);
	\draw (3.7,1.7) -- (4,2) -- (5,2) -- (6.3,.7);
	\draw (6.7,.3) -- (7,0);
	\draw (3,2) -- (4.3,.7);
	\draw (4.7,.3) -- (5,0) -- (6,0) -- (7,1);
\end{scope}
\draw[->] (3.5,1.5) -- (3.9,1.1);
\draw[->] (3.7,1.7) -- (3.9,1.9);
\draw[->] (4.5,.5) -- (4.1,.1);
\draw[->] (4.7,.3) -- (4.9,.1);
\draw[->] (5.5,1.5) -- (5.9,1.1);
\draw[->] (5.3,1.3) -- (5.1,1.1);
\draw[->] (6.5,.5) -- (6.9,.9);
\draw[->] (6.7,.3) -- (6.9,.1);

\end{scope}

\begin{scope}[scale=.33, yshift = -1cm, xshift = 10cm]
\draw (1,-.5) rectangle (3,2.5);
\draw (7,-.5) rectangle (9,2.5);
\draw (0,0) -- (1,0);
\draw (9,2) -- (10,2);
\begin{scope}[rounded corners = 1mm]
	\draw[->] (3,2) -- (3.5,1.6) -- (4,2) -- (5,2) -- (5.4,1.5) -- (5,1) -- (4.6,.5) -- (5,0) -- (6,0) -- (6.5,.4) -- (7,0);
	\draw[->] (3,1) -- (3.5,1.4) -- (4,1) -- (4.4,.5) -- (4,0) -- (3,0);
	\draw[->] (7,2) -- (6,2) -- (5.6,1.5) -- (6,1) -- (6.5,.6) -- (7,1);
\end{scope}
\draw[densely dashed] (1,0) -- (3,2);
\draw[densely dashed, rounded corners=1mm] (3,1) -- (2.6,.5) -- (3,0);
\draw[densely dashed] (7,0) -- (9,2);
\draw[densely dashed, rounded corners=1mm] (7,2) -- (7.4,1.5) -- (7,1);
\end{scope}

\draw[->] (7,-.25) -- (7.5,-.25);
\draw (7.25,.25) node{\small{2}};

\begin{scope}[xshift = 8.5 cm]

\draw(-1,-1.25) rectangle (5,.75);
\draw(2,-1) node{${\scriptstyle +-}$};
\begin{scope}[scale=.33, xshift = -2cm, yshift = -1cm]
\draw[white] (-.2,-.7) rectangle (8.2,2.7);
\draw (1,-.5) rectangle (3,2.5);
\draw (2,1) node{$R$};
\draw (6,1) node[rotate = 180]{$\overline{R}$};
\draw (5,-.5) rectangle (7,2.5);
\draw (0,0) -- (1,0);
\draw (7,2) -- (8,2);
\begin{scope}[rounded corners = 1mm]
	\draw (3,0) -- (4.3,1.3);
	\draw (4.7,1.7) -- (5,2);
	\draw (3,1) -- (3.3,.7);
	\draw (3,2) -- (4,2) -- (5,1);
	\draw (3.7,.3) -- (4,0) -- (5,0);
\end{scope}

\draw[->] (3.5,.5) -- (3.1,.1);
\draw[->] (3.7,.3) -- (3.9,.1);
\draw[->] (4.5,1.5) -- (4.9,1.1);
\draw[->] (4.3,1.3) -- (4.1,1.1);
\end{scope}

\begin{scope}[scale=.33, xshift = 6.5cm, yshift = -1cm]
\draw (1,-.5) rectangle (3,2.5);
\draw (5,-.5) rectangle (7,2.5);
\draw (0,0) -- (1,0);
\draw (7,2) -- (8,2);
\begin{scope}[rounded corners = 1mm]
	\draw[->] (3,1) -- (3.4,.5) -- (3,0);
	\draw[->] (5,2) -- (4.6,1.5) -- (5,1);
	\draw[->] (3,2) -- (4,2) -- (4.4,1.5) -- (4,1) -- (3.6,.5) -- (4,0) -- (5,0);
\end{scope}
\draw[densely dashed] (1,0) -- (3,2);
\draw[densely dashed, rounded corners=1mm] (3,0) -- (2.6,.5) -- (3,1);
\draw[densely dashed] (5,0) -- (7,2);
\draw[densely dashed, rounded corners=1mm] (5,1) -- (5.4,1.5) -- (5,2);
\end{scope}
\end{scope}
\end{scope}


\begin{scope}[yshift=-5cm]

\draw(-1,-1.25) rectangle (7,.75);
\draw(3,-1) node{${\scriptstyle -+--++-+}$};
\begin{scope}[scale=.33, yshift =  -1cm, xshift = -3.25cm]
\draw (1,-.5) rectangle (3,2.5);
\draw (2,1) node{$R$};
\draw (10,1) node[rotate = 180]{$\overline{R}$};
\draw (9,-.5) rectangle (11,2.5);
\draw (0.5,0) -- (1,0);
\draw (11,2) -- (11.5,2);
\begin{scope}[rounded corners = 1mm]
	\draw (3,0) -- (4,0) -- (5,1) -- (5.3,.7);
	\draw (5.7,.3) -- (6,0) -- (8,0) -- (9,1);
	\draw (3,1) -- (3.3,1.3);
	\draw (3.7,1.7) -- (4,2) -- (6,2) -- (7,1) -- (7.3,1.3);
	\draw (7.7,1.7) -- (8,2) -- (9,2);
	\draw (3,2) -- (4.3,.7);
	\draw (4.7,.3) -- (5,0) -- (6.3,1.3);
	\draw (6.7,1.7) -- (7,2) -- (8.3,.7);
	\draw (8.7,.3) -- (9,0);
\end{scope}
\draw[->] (3.5,1.5) -- (3.1,1.9);
\draw[->] (3.7,1.7) -- (3.9,1.9);
\draw[->] (4.5,.5) -- (4.9,.9);
\draw[->] (4.3,.7) -- (4.1,.9);
\draw[->] (5.5,.5) -- (5.1,.1);
\draw[->] (5.7,.3) -- (5.9,.1);
\draw[->] (6.5,1.5) --(6.9,1.1);
\draw[->] (6.3,1.3) -- (6.1,1.1);
\draw[->] (7.5,1.5) -- (7.1,1.9);
\draw[->] (7.7,1.7) -- (7.9,1.9);
\draw[->] (8.5,.5) -- (8.9,.9);
\draw[->] (8.3,.7) -- (8.1,.9);
\end{scope}

\begin{scope}[scale=.33, yshift = -1cm, xshift = 9cm]
\draw (1,-.5) rectangle (3,2.5);
\draw (9,-.5) rectangle (11,2.5);
\draw (0.5,0) -- (1,0);
\draw (11,2) -- (11.5,2);
\begin{scope}[rounded corners = 1mm]
\draw[->] (3,0) -- (4,0) -- (4.4,.5) -- (4,1) -- (3.6,1.5) -- (4,2) -- (6,2) -- (6.4,1.5) -- (6,1) --(5.6,.5) -- (6,0) -- (8,0) -- (8.4,.5) -- (8,1) -- (7.6,1.5) -- (8,2) -- (9,2);
\draw[->] (3,1) -- (3.4,1.5) -- (3,2);
\draw[->] (9,0) -- (8.6,.5) -- (9,1);
\draw[->] (5,1) arc (90:-270:.4cm and .5cm);
\draw[->] (7,2) arc (90:450:.4cm and .5cm);
\end{scope}
\draw[densely dashed] (1,0) -- (3,0);
\draw[densely dashed, rounded corners =1mm] (3,1) -- (2.6,1.5) -- (3,2);
\draw[densely dashed] (9,2) -- (11,2);
\draw[densely dashed, rounded corners =1mm] (9,1) -- (9.4,.5) -- (9,0);
\end{scope}

\draw[->] (7,-.25) -- (7.5,-.25);
\draw (7.25,.25) node{\small{3A}};

\begin{scope}[xshift = 8.5 cm]

\draw(-1,-1.25) rectangle (5,.75);
\draw(2,-1) node{${\scriptstyle -+}$};
\begin{scope}[scale=.33, xshift = -2cm, yshift = -1cm]
\draw (1,-.5) rectangle (3,2.5);
\draw (2,1) node{$R$};
\draw (6,1) node[rotate = 180]{$\overline{R}$};
\draw (5,-.5) rectangle (7,2.5);
\draw (0,0) -- (1,0);
\draw (7,2) -- (8,2);
\begin{scope}[rounded corners = 1mm]
	\draw (3,0) -- (4,0) -- (5,1);
	\draw (3,1) -- (3.3,1.3);
	\draw (3.7,1.7) -- (4,2) -- (5,2);
	\draw (3,2) -- (4.3,.7);
	\draw (4.7,.3) -- (5,0);
\end{scope}

\draw[->] (3.5,1.5) -- (3.1,1.9);
\draw[->] (3.7,1.7) -- (3.9,1.9);
\draw[->] (4.5,.5) -- (4.9,.9);
\draw[->] (4.3,.7) -- (4.1,.9);

\end{scope}

\begin{scope}[scale=.33, xshift = 6.5cm, yshift = -1cm]
\draw (1,-.5) rectangle (3,2.5);
\draw (5,-.5) rectangle (7,2.5);
\draw (0,0) -- (1,0);
\draw (7,2) -- (8,2);
\begin{scope}[rounded corners = 1mm]
	\draw[->] (3,1) -- (3.4,1.5) -- (3,2);
	\draw[->] (5,0) -- (4.6,.5) -- (5,1);
	\draw[->] (3,0) -- (4,0) -- (4.4,.5) -- (4,1) -- (3.6,1.5) -- (4,2) -- (5,2);
\end{scope}
\draw[densely dashed] (1,0) -- (3,0);
\draw[densely dashed, rounded corners=1mm] (3,2) -- (2.6,1.5) -- (3,1);
\draw[densely dashed] (5,2) -- (7,2);
\draw[densely dashed, rounded corners=1mm] (5,1) -- (5.4,.5) -- (5,0);
\end{scope}
\end{scope}
\end{scope}


\begin{scope}[yshift=-7.5cm]

\draw(-1,-1.25) rectangle (7,.75);
\draw(3,-1) node{${\scriptstyle --+--++-++}$};
\begin{scope}[scale=.33, yshift =  -1cm, xshift = -3.25cm]
\draw (1,-.5) rectangle (3,2.5);
\draw (2,1) node{$R$};
\draw (10,1) node[rotate = 180]{$\overline{R}$};
\draw (9,-.5) rectangle (11,2.5);
\draw (.5,0) -- (1,0);
\draw (11,2) -- (11.5,2);
\begin{scope}[rounded corners = 1mm]
	\draw (3,0) -- (4,1) -- (4.3,.7);
	\draw (3.7,.3) -- (4,0) -- (5,1) -- (5.3,.7);
	\draw (5.7,.3) -- (6,0) -- (9,0);
	\draw (3,1) -- (3.3,.7);
	\draw (7.7,1.7) -- (8,2) -- (9,1);
	\draw (3,2) -- (6,2) -- (7,1) -- (7.3,1.3);
	\draw (4.7,.3) -- (5,0) -- (6.3,1.3);
	\draw (6.7,1.7) -- (7,2) -- (8,1) -- (8.3,1.3);
	\draw (8.7,1.7) -- (9,2);
\end{scope}
\draw[->] (3.5,.5) -- (3.1,.1);
\draw[->] (3.7,.3) -- (3.9,.1);
\draw[->] (4.5,.5) -- (4.9,.9);
\draw[->] (4.3,.7) -- (4.1,.9);
\draw[->] (5.5,.5) -- (5.1,.1);
\draw[->] (5.7,.3) -- (5.9,.1);
\draw[->] (6.5,1.5) --(6.9,1.1);
\draw[->] (6.3,1.3) -- (6.1,1.1);
\draw[->] (7.5,1.5) -- (7.1,1.9);
\draw[->] (7.7,1.7) -- (7.9,1.9);
\draw[->] (8.5,1.5) -- (8.9,1.1);
\draw[->] (8.3,1.3) -- (8.1,1.1);

\end{scope}

\begin{scope}[scale=.33, yshift = -1cm, xshift = 9cm]
\draw (1,-.5) rectangle (3,2.5);
\draw (9,-.5) rectangle (11,2.5);
\draw (0.5,0) -- (1,0);
\draw (11,2) -- (11.5,2);
\begin{scope}[rounded corners = 1mm]
\draw[->] (3,2) -- (6,2) -- (6.4,1.5) -- (6,1) -- (5.6,.5) -- (6,0) -- (9,0);
\draw[->] (3,1) -- (3.4,.5) -- (3,0);
\draw[->] (9,2) -- (8.6,1.5) -- (9,1);
\draw[->] (5,1) arc (90:-270:.4cm and .5cm);
\draw[->] (4,1) arc (90:450:.4cm and .5cm);
\draw[->] (7,2) arc (90:450:.4cm and .5cm);
\draw[->] (8,2) arc (90:-270:.4cm and .5cm);
\end{scope}
\draw[densely dashed] (1,0) -- (3,2);
\draw[densely dashed, rounded corners =1mm] (3,0) -- (2.6,.5) -- (3,1);
\draw[densely dashed] (9,0) -- (11,2);
\draw[densely dashed, rounded corners =1mm] (9,1) -- (9.4,1.5) -- (9,2);
\end{scope}

\draw[->] (7,-.25) -- (7.5,-.25);
\draw (7.25,.25) node{\small{3B}};

\begin{scope}[xshift = 8.5 cm]

\draw(-1,-1.25) rectangle (5,.75);
\draw(2,-1) node{${\scriptstyle --++}$};
\begin{scope}[scale=.33, xshift = -2cm, yshift = -1cm]
\draw[white] (-.2,-.7) rectangle (8.2,2.7);
\draw (1,-.5) rectangle (3,2.5);
\draw (2,1) node{$R$};
\draw (6,1) node[rotate = 180]{$\overline{R}$};
\draw (5,-.5) rectangle (7,2.5);
\draw (0,0) -- (1,0);
\draw (7,2) -- (8,2);
\begin{scope}[rounded corners = 1mm]
	\draw (3,0) -- (4.3,1.3);
	\draw (4.7,1.7) -- (5,2);
	\draw (3,1) -- (3.3,.7);
	\draw (3,2) -- (4,2) -- (5,1);
	\draw (3.7,.3) -- (4,0) -- (5,0);
\end{scope}

\draw[->] (3.5,.5) -- (3.1,.1);
\draw[->] (3.7,.3) -- (3.9,.1);
\draw[->] (4.5,1.5) -- (4.9,1.1);
\draw[->] (4.3,1.3) -- (4.1,1.1);
\end{scope}

\begin{scope}[scale=.33, xshift = 6.5cm, yshift = -1cm]
\draw (1,-.5) rectangle (3,2.5);
\draw (5,-.5) rectangle (7,2.5);
\draw (0,0) -- (1,0);
\draw (7,2) -- (8,2);
\begin{scope}[rounded corners = 1mm]
	\draw[->] (3,1) -- (3.4,.5) -- (3,0);
	\draw[->] (5,2) -- (4.6,1.5) -- (5,1);
	\draw[->] (3,2) -- (4,2) -- (4.4,1.5) -- (4,1) -- (3.6,.5) -- (4,0) -- (5,0);
\end{scope}
\draw[densely dashed] (1,0) -- (3,2);
\draw[densely dashed, rounded corners=1mm] (3,0) -- (2.6,.5) -- (3,1);
\draw[densely dashed] (5,0) -- (7,2);
\draw[densely dashed, rounded corners=1mm] (5,1) -- (5.4,1.5) -- (5,2);
\end{scope}
\end{scope}
\end{scope}


\begin{scope}[yshift=-10cm]

\draw(-1,-1.25) rectangle (7,.75);
\draw(3,-1) node{${\scriptstyle -++-+--+}$};
\begin{scope}[scale=.33, yshift =  -1cm, xshift = -3.25cm]
\draw (1,-.5) rectangle (3,2.5);
\draw (2,1) node{$R$};
\draw (10,1) node[rotate = 180]{$\overline{R}$};
\draw (9,-.5) rectangle (11,2.5);
\draw (0.5,0) -- (1,0);
\draw (11,2) -- (11.5,2);
\begin{scope}[rounded corners = 1mm]
	\draw (3,0) -- (6,0) -- (7,1) -- (7.3,.7);
	\draw (7.7,.3) -- (8,0) -- (9,1);
	\draw (3,1) -- (3.3,1.3);
	\draw (3.7,1.7) -- (4,2) -- (5,1) -- (5.3,1.3);
	\draw (5.7,1.7) -- (6,2) -- (9,2);
	\draw (3,2) -- (4,1) -- (4.3,1.3);
	\draw (4.7,1.7) -- (5,2) -- (6.3,.7);
	\draw (6.7,.3) -- (7,0) -- (8,1) -- (8.3,.7);
	\draw (8.7,.3) -- (9,0);
\end{scope}
\draw[->] (3.5,1.5) -- (3.1,1.9);
\draw[->] (3.7,1.7) -- (3.9,1.9);
\draw[->] (4.5,1.5) -- (4.9,1.1);
\draw[->] (4.3,1.3) -- (4.1,1.1);
\draw[->] (5.5,1.5) -- (5.1,1.9);
\draw[->] (5.7,1.7) -- (5.9,1.9);
\draw[->] (6.5,.5) --(6.9,.9);
\draw[->] (6.3,.7) -- (6.1,.9);
\draw[->] (7.5,.5) -- (7.1,.1);
\draw[->] (7.7,.3) -- (7.9,.1);
\draw[->] (8.5,.5) -- (8.9,.9);
\draw[->] (8.3,.7) -- (8.1,.9);
\end{scope}

\begin{scope}[scale=.33, yshift = -1cm, xshift = 9cm]
\draw (1,-.5) rectangle (3,2.5);
\draw (9,-.5) rectangle (11,2.5);
\draw (0.5,0) -- (1,0);
\draw (11,2) -- (11.5,2);
\begin{scope}[rounded corners = 1mm]
\draw[->] (3,0) -- (6,0) -- (6.4,.5) -- (6,1) -- (5.6,1.5) -- (6,2) -- (9,2);
\draw[->] (3,1) -- (3.4,1.5) -- (3,2);
\draw[->] (9,0) -- (8.6,.5) -- (9,1);
\draw[->] (5,2) arc (90:450:.4cm and .5cm);
\draw[->] (4,2) arc (90:-270:.4cm and .5cm);
\draw[->] (7,1) arc (90:-270:.4cm and .5cm);
\draw[->] (8,1) arc (90:450:.4cm and .5cm);
\end{scope}
\draw[densely dashed] (1,0) -- (3,0);
\draw[densely dashed, rounded corners =1mm] (3,1) -- (2.6,1.5) -- (3,2);
\draw[densely dashed] (9,2) -- (11,2);
\draw[densely dashed, rounded corners =1mm] (9,1) -- (9.4,.5) -- (9,0);
\end{scope}

\draw[->] (7,-.25) -- (7.5,-.25);
\draw (7.25,.25) node{\small{4A}};

\begin{scope}[xshift = 8.5 cm]

\draw(-1,-1.25) rectangle (5,.75);
\draw(2,-1) node{${\scriptstyle -+}$};
\begin{scope}[scale=.33, xshift = -2cm, yshift = -1cm]
\draw (1,-.5) rectangle (3,2.5);
\draw (2,1) node{$R$};
\draw (6,1) node[rotate = 180]{$\overline{R}$};
\draw (5,-.5) rectangle (7,2.5);
\draw (0,0) -- (1,0);
\draw (7,2) -- (8,2);
\begin{scope}[rounded corners = 1mm]
	\draw (3,0) -- (4,0) -- (5,1);
	\draw (3,1) -- (3.3,1.3);
	\draw (3.7,1.7) -- (4,2) -- (5,2);
	\draw (3,2) -- (4.3,.7);
	\draw (4.7,.3) -- (5,0);
\end{scope}

\draw[->] (3.5,1.5) -- (3.1,1.9);
\draw[->] (3.7,1.7) -- (3.9,1.9);
\draw[->] (4.5,.5) -- (4.9,.9);
\draw[->] (4.3,.7) -- (4.1,.9);

\end{scope}

\begin{scope}[scale=.33, xshift = 6.5cm, yshift = -1cm]
\draw (1,-.5) rectangle (3,2.5);
\draw (5,-.5) rectangle (7,2.5);
\draw (0,0) -- (1,0);
\draw (7,2) -- (8,2);
\begin{scope}[rounded corners = 1mm]
	\draw[->] (3,1) -- (3.4,1.5) -- (3,2);
	\draw[->] (5,0) -- (4.6,.5) -- (5,1);
	\draw[->] (3,0) -- (4,0) -- (4.4,.5) -- (4,1) -- (3.6,1.5) -- (4,2) -- (5,2);
\end{scope}
\draw[densely dashed] (1,0) -- (3,0);
\draw[densely dashed, rounded corners=1mm] (3,2) -- (2.6,1.5) -- (3,1);
\draw[densely dashed] (5,2) -- (7,2);
\draw[densely dashed, rounded corners=1mm] (5,1) -- (5.4,.5) -- (5,0);
\end{scope}
\end{scope}
\end{scope}


\begin{scope}[yshift=-12.5cm]

\draw(-1,-1.25) rectangle (7,.75);
\draw(3,-1) node{${\scriptstyle --++-+--++}$};
\begin{scope}[scale=.33, yshift =  -1cm, xshift = -3.25cm]
\draw (1,-.5) rectangle (3,2.5);
\draw (2,1) node{$R$};
\draw (10,1) node[rotate = 180]{$\overline{R}$};
\draw (9,-.5) rectangle (11,2.5);
\draw (0.5,0) -- (1,0);
\draw (11,2) -- (11.5,2);
\begin{scope}[rounded corners = 1mm]
	\draw (3,0) -- (4.3,1.3);
	\draw (4.7,1.7) -- (5,2) -- (6.3,.7);
	\draw (6.7,.3) -- (7,0) -- (8.3,1.3);
	\draw (8.7,1.7) -- (9,2);
	\draw (3,1) -- (3.3,.7);
	\draw (3.7,.3) -- (4,0) -- (6,0) -- (7,1) -- (7.3,.7);
	\draw (7.7,.3) -- (8,0) -- (9,0);
	\draw (3,2) -- (4,2) -- (5,1) -- (5.3,1.3);
	\draw (5.7,1.7) -- (6,2) -- (8,2) -- (9,1);
\end{scope}
\draw[->] (3.5,.5) -- (3.1,.1);
\draw[->] (3.7,.3) -- (3.9,.1);
\draw[->] (4.5,1.5) -- (4.9,1.1);
\draw[->] (4.3,1.3) -- (4.1,1.1);
\draw[->] (5.5,1.5) -- (5.1,1.9);
\draw[->] (5.7,1.7) -- (5.9,1.9);
\draw[->] (6.5,.5) --(6.9,.9);
\draw[->] (6.3,.7) -- (6.1,.9);
\draw[->] (7.5,.5) -- (7.1,.1);
\draw[->] (7.7,.3) -- (7.9,.1);
\draw[->] (8.5,1.5) -- (8.9,1.1);
\draw[->] (8.3,1.3) -- (8.1,1.1);

\end{scope}

\begin{scope}[scale=.33, yshift = -1cm, xshift = 9cm]
\draw (1,-.5) rectangle (3,2.5);
\draw (9,-.5) rectangle (11,2.5);
\draw (0.5,0) -- (1,0);
\draw (11,2) -- (11.5,2);
\begin{scope}[rounded corners = 1mm]
\draw[->] (3,2) -- (4,2) -- (4.4,1.5) -- (4,1) -- (3.6,.5) -- (4,0) -- (6,0) -- (6.4,.5) -- (6,1) -- (5.6,1.5) -- (6,2) -- (8,2) -- (8.4,1.5) -- (8,1) -- (7.6,.5) -- (8,0) -- (9,0);
\draw[->] (3,1) -- (3.4,.5) -- (3,0);
\draw[->] (9,2) -- (8.6,1.5) -- (9,1);
\draw[->] (5,2) arc (90:450:.4cm and .5cm);
\draw[->] (7,1) arc (90:-270:.4cm and .5cm);
\end{scope}
\draw[densely dashed] (1,0) -- (3,2);
\draw[densely dashed, rounded corners =1mm] (3,0) -- (2.6,.5) -- (3,1);
\draw[densely dashed] (9,0) -- (11,2);
\draw[densely dashed, rounded corners =1mm] (9,1) -- (9.4,1.5) -- (9,2);
\end{scope}

\draw[->] (7,-.25) -- (7.5,-.25);
\draw (7.25,.25) node{\small{4B}};

\begin{scope}[xshift = 8.5 cm]

\draw(-1,-1.25) rectangle (5,.75);
\draw(2,-1) node{${\scriptstyle --++}$};
\begin{scope}[scale=.33, xshift = -2cm, yshift = -1cm]
\draw[white] (-.2,-.7) rectangle (8.2,2.7);
\draw (1,-.5) rectangle (3,2.5);
\draw (2,1) node{$R$};
\draw (6,1) node[rotate = 180]{$\overline{R}$};
\draw (5,-.5) rectangle (7,2.5);
\draw (0,0) -- (1,0);
\draw (7,2) -- (8,2);
\begin{scope}[rounded corners = 1mm]
	\draw (3,0) -- (4.3,1.3);
	\draw (4.7,1.7) -- (5,2);
	\draw (3,1) -- (3.3,.7);
	\draw (3,2) -- (4,2) -- (5,1);
	\draw (3.7,.3) -- (4,0) -- (5,0);
\end{scope}

\draw[->] (3.5,.5) -- (3.1,.1);
\draw[->] (3.7,.3) -- (3.9,.1);
\draw[->] (4.5,1.5) -- (4.9,1.1);
\draw[->] (4.3,1.3) -- (4.1,1.1);
\end{scope}

\begin{scope}[scale=.33, xshift = 6.5cm, yshift = -1cm]
\draw (1,-.5) rectangle (3,2.5);
\draw (5,-.5) rectangle (7,2.5);
\draw (0,0) -- (1,0);
\draw (7,2) -- (8,2);
\begin{scope}[rounded corners = 1mm]
	\draw[->] (3,1) -- (3.4,.5) -- (3,0);
	\draw[->] (5,2) -- (4.6,1.5) -- (5,1);
	\draw[->] (3,2) -- (4,2) -- (4.4,1.5) -- (4,1) -- (3.6,.5) -- (4,0) -- (5,0);
\end{scope}
\draw[densely dashed] (1,0) -- (3,2);
\draw[densely dashed, rounded corners=1mm] (3,0) -- (2.6,.5) -- (3,1);
\draw[densely dashed] (5,0) -- (7,2);
\draw[densely dashed, rounded corners=1mm] (5,1) -- (5.4,1.5) -- (5,2);
\end{scope}
\end{scope}
\end{scope}

\end{tikzpicture}\]

    \caption{Center strings, alternating diagrams and partial Seifert states corresponding to the cases in Lemma \ref{lemma:tprecur}. The center strings appear along the bottom of each rectangle, and the case number appears above each arrow.}
    \label{fig:tp}
\end{figure}
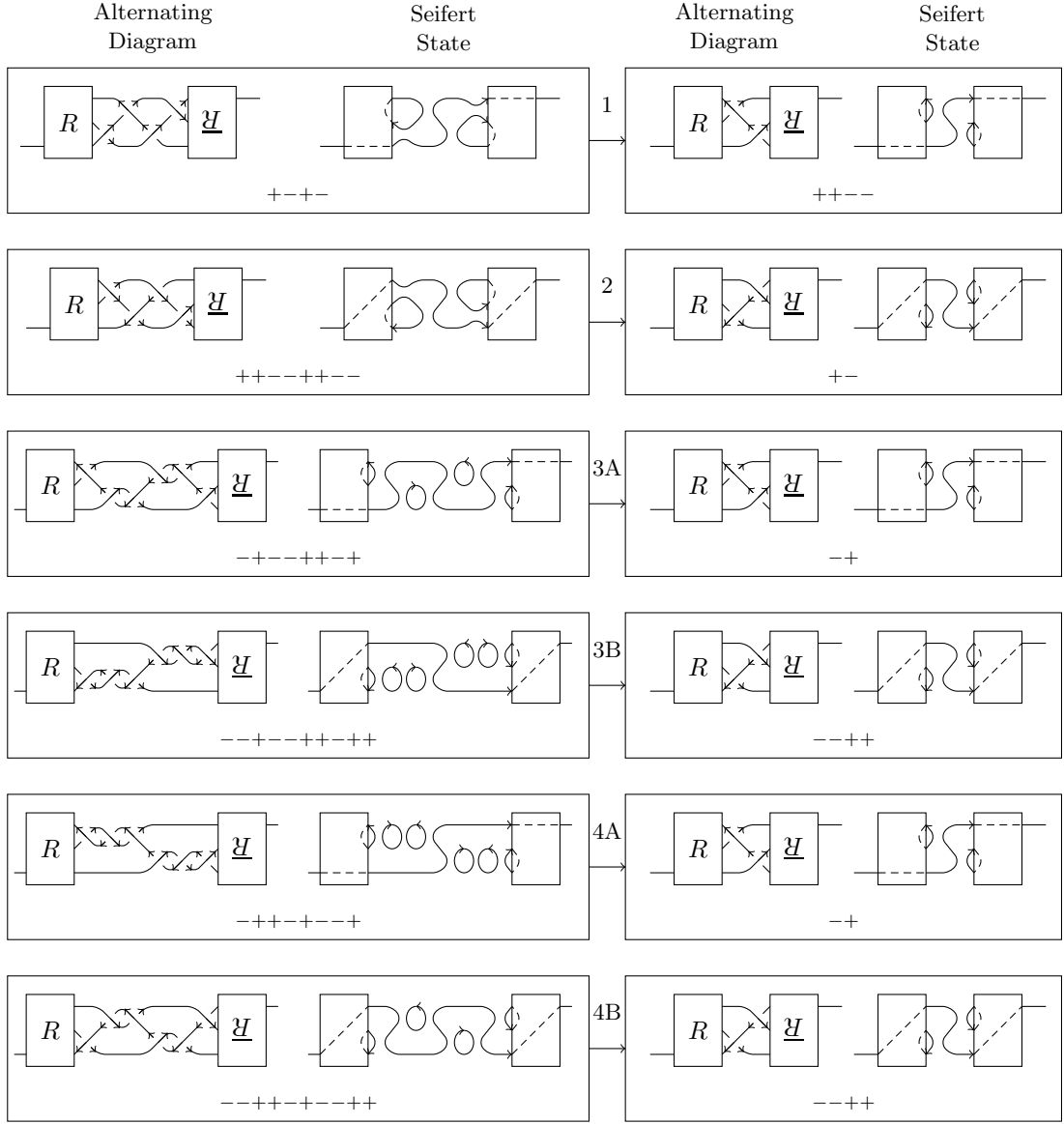

\subsection{Other formulas} Combining Lemmas \ref{lemma:tcg} and \ref{lemma:tp(c,g)} yields a proof of Theorem \ref{thm:tbarformula}.

\begin{proof}[Proof of Theorem \ref{thm:tbarformula}]
Theorem \ref{thm:list} implies that $t(c,g)+t_p(c,g)=2\bar{t}(c,g)$, twice the number of $2$-bridge knots with crossing number $c$ and genus $g$. Therefore Lemmas \ref{lemma:tcg} and \ref{lemma:tp(c,g)} imply
\[ \bar{t}(c,g) = \frac{1}{2}\left((-1)^{c'-g-1} \sum_{n=0}^{c'-g-1}(-1)^n{n+g-1\choose n} + (-1)^{c-1}\sum_{n=0}^{c-2g-1}(-1)^{n}{n+2g-1 \choose n}\right).\]
\end{proof}

Table \ref{tab:tbar} provides the number $\bar{t}(c,g)$ of $2$-bridge knots with crossing number $c$ and genus $g$ up to crossing number 20 using the formula from Theorem \ref{thm:tbarformula}.

\begin{center}
\begin{table}[h]
\begin{tabular}{|c|c c c c c c c c c|}
\hline
$c \setminus g$ & $1$ & $2$ & $3$ & $4$ & $5$ & $6$ & $7$ & $8$ & $9$\\ [0.5ex] 
 \hline\hline
 3 & 1 &  &  &  &  &  &  &  &  \\ 
 \hline
  4 & 1 &  &  &  &  &  &  &  &  \\
\hline
 5 & 1 & 1 &  &  &  &  &  &  &  \\
 \hline
 6 & 1 & 2 &  &  &  &  &  &  &  \\
 \hline
 7 & 2 & 4 & 1 &  &  &  &  &  &  \\
 \hline
 8 & 2 & 7 & 3 &  &  &  &  &  &  \\
 \hline
 9 & 2 & 12 & 9 & 1 &  &  &  &  &  \\
 \hline
 10 & 2 & 18 & 21 & 4 &  &  &  &  &  \\
 \hline
 11 & 3 & 26 & 45 & 16 & 1 &  &  &  &  \\
 \hline
 12 & 3 & 36 & 85 & 47 & 5 &  &  &  &  \\
 \hline
 13 & 3 & 49 & 151 & 123 & 25 & 1 &  &  &  \\
 \hline
 14 & 3 & 64 & 251 & 280 & 89 & 6 &  &  &  \\
 \hline
 15 & 4 & 82 & 400 & 588 & 276 & 36 & 1 &  &  \\
 \hline
 16 & 4 & 103 & 610 & 1141 & 736 & 151 & 7 &  &  \\
 \hline
 17 & 4 & 128 & 904 & 2094 & 1784 & 542 & 49 & 1 &  \\
 \hline
 18 & 4 & 156 & 1294 & 3648 & 3960 & 1658 & 237 & 8 &  \\
 \hline
 19 & 5 & 188 & 1814 & 6104 & 8230 & 4558 & 967 & 64 & 1 \\
 \hline
 20 & 5 & 224 & 2486 & 9842 & 16126 & 11394 & 3339 & 351 & 9 \\
 \hline
\end{tabular}
\caption{The $\bar{t}(c,g)$ values up to crossing number 20 as determined using the formula  from Theorem \ref{thm:tbarformula}.}
\label{tab:tbar}
\end{table}
\end{center}

In order to find the median, mode, and variance of the genus of 2-bridge knots in Sections \ref{sec:Mode} and
\ref{sec:variance}, we use the following additional facts about $t(c,g)$, $t_p(c,g)$, and $\bar t(c,g)$.

\newpage

\begin{theorem}
\label{thm:3formulas}
The following relationships between $t(c,g)$, $t_p(c,g)$, and $\bar t(c,g)$ hold.
\begin{enumerate}
    \item If $c\geq5$ and $g \geq 2$, then $t_p(c,g) = t_p(c-2,g) + t_p(c-2,g-1)$.
    \item If $c\geq 3$, then $t_p(2c,2g) = t(c,g)$.
    \item If $c\geq 5$ and $g\geq 2$, then $\bar t(c,g) = \bar t(c-2,g)+\bar t(c-2,g-1)+t_p(2c-4,2g-1)$.
\end{enumerate}
    
\end{theorem}
\begin{proof}
    In order to prove statement (1), we proceed by induction using the formula from Lemma \ref{lemma:tp(c,g)}\[t_p(c,g)=(-1)^{c'-g-1}\sum_{n=0}^{c'-g-1}(-1)^n{n+g-1 \choose n},\] where $c'=\frac{c}{2}$ if $c$ is even $c'=\frac{c+1}{2}$ if c is odd.


For the base cases, the statement holds when $c=5$ and $g>1$, as direct computation confirms that
\begin{align*}
    t_p(5,2) = & \; 1 = 0 + 1 = t_p(3,2)+t_p(3,1),\\
    t_p(5,k) = & \; 0 = 0 + 0 = t_p(3,k) + t_p(3,k-1)~\text{for $k\geq 3$},\\
    t_p(6,2) = & \; 1 = 0 + 1 = t_p(4,2) + t_p(4,1),~\text{and}\\
    t_p(6,k) = & \; 0 = 0 + 0 = t_p(4,k) + t_p(4,k-1)~\text{for $k\geq 3$}.
    \end{align*}

For the inductive step, assume this identity holds for $c'-1$, that is, assume
\begin{align*}
(-1)^{c'-g-2}\sum_{n=0}^{c'-g-2}(-1)^n{n+g-1 \choose n} = &\; (-1)^{c'-g-3}\sum_{n=0}^{c'-g-3}(-1)^n{n+g-1 \choose n}\\
&\; + 
(-1)^{c'-g-2}\sum_{n=0}^{c'-g-2}(-1)^n{n+g-2 \choose n}
\end{align*}
for all $g>1$. Then
\begin{align*}
t_p(c,g) = & \; t_p(c-2,g) + t_p(c-2,g-1) \\
= &\; (-1)^{c'-g-2}\sum_{n=0}^{c'-g-2}(-1)^n{n+g-1 \choose n} + 
(-1)^{c'-g-1}\sum_{n=0}^{c'-g-1}(-1)^n{n+g-2 \choose n} \\
 = & \; (-1)^{c'-g-2}\sum_{n=0}^{c'-g-3}(-1)^n{n+g-1 \choose n} +{c'-3 \choose c'-g-2}\\
 & \; + (-1)^{c'-g-1}\sum_{n=0}^{c'-g-2}(-1)^n{n+g-2 \choose n} +{c'-3 \choose c'-g-1}\\
 = & \;-\left((-1)^{c'-g-3}\sum_{n=0}^{c'-g-3}(-1)^n{n+g-1 \choose n} + (-1)^{c'-g-2}\sum_{n=0}^{c'-g-2}(-1)^n{n+g-2 \choose n}\right) \\ 
& \; + {c'-3 \choose c'-g-1} + {c'-3 \choose c'-g-2}\\
= &\; (-1)^{c'-g-1}\sum_{n=0}^{c'-g-2}(-1)^n{n+g-1 \choose n} +  {c'-2 \choose c-g-1}\\
 = & \; (-1)^{c'-g-1}\sum_{n=0}^{c'-g-1}(-1)^n{n+g-1 \choose n},
\end{align*}
where the fifth equality follows from the inductive hypothesis and the relation ${n \choose k} = {n-1 \choose k-1} + {n-1 \choose k}$.

Statement (2) follows the formulas for $t(c,g)$ and $t_p(c,g)$ in Lemmas \ref{lemma:tcg} and \ref{lemma:tp(c,g)} via
\[t_p\left(2c,2g\right) = \left(-1\right)^{\frac{2c}{2}-2g-1}\sum_{n=0}^{\frac{2c}{2}-2g-1}\left(-1\right)^n{n+2g-1 \choose n} = \left(-1\right)^{c-1}\sum_{n=0}^{c-2g-1}\left(-1\right)^n{n+2g-1 \choose n} = t\left(c,g\right).\]


To prove statement (3), we express $\bar t(c,g)$ as $\bar t(c,g) =  \frac{1}{2} [t(c,g)+t_p(c,g)]$ and use the recursive formulas for $t(c,g)$ and $t_p(c,g)$ from Lemma \ref{lemma:trecur} and part (1) of this theorem. When $g\neq 1$, we have
\begin{align*}
\bar t(c,g) = & \; \frac{1}{2} [t(c-1,g)+t(c-2,g)+t(c-2,g-1)+t(c-3,g-1) -t(c-3,g)\\
& \;+t_p(c-2,g)+t_p(c-2,g-1)]\\
 =  &\;  \; \bar t(c-2,g)+\bar t(c-2,g-1) +  \frac{1}{2}[t(c-1,g)+t(c-3,g-1)-t(c-3,g)].
 \end{align*}

We now only have to show $t(c-1,g)+t(c-3,g-1)-t(c-3,g) = 2t_p(2c-4,2g-1)$. Lemma \ref{lemma:tcg} implies that
\begin{align*}
&t(c-1,g)+t(c-3,g-1)-t(c-3,g)\\
= & \; (-1)^{c}\sum_{n=0}^{c-2g-2}(-1)^n{n+2g-3 \choose n} 
+ (-1)^{c}\sum_{n=0}^{c-2g-2}(-1)^n{n+2g-2 \choose n}
- (-1)^{c}\sum_{n=0}^{c-2g-3}(-1)^n{n+2g-2 \choose n} \\
= & \; (-1)^{c}\sum_{n=0}^{c-2g-2}(-1)^n\left[{n+2g-3 \choose n} + {n+2g-2 \choose n}\right]
+ (-1)^{c}\sum_{n=0}^{c-2g-3}(-1)^{n+1}{(n+1)+2g-3 \choose n} \\
= & \; (-1)^{c}\sum_{n=0}^{c-2g-2}(-1)^n\left[{n+2g-2 \choose n} + {n+2g-3 \choose n}\right]
+ (-1)^{c}\sum_{n=1}^{c-2g-2}(-1)^{n}{n+2g-3 \choose n-1} \\
= & \; (-1)^{c}\sum_{n=0}^{c-2g-2}(-1)^n\left[{n+2g-2 \choose n} + {n+2g-3 \choose n} + {n+2g-3 \choose n-1}\right] \\
= & \; 2 (-1)^{c}\sum_{n=0}^{c-2g-2}(-1)^n{n+2g-2 \choose n} \\
= & \; 2 t_p(2c-4,2g-1),
\end{align*}
where the second inequality follows from combining the first two sums, the third equality follows from changing the index of the last sum, the fourth equality follows from combining, and the fifth equality follows from the relation $\binom{n}{k}=\binom{n-1}{k}+\binom{n-1}{k-1}$.
\end{proof}

It will be useful in our proofs to have formulas for $t_p(c,g)$, $t(c,g)$, and $\bar{t}(c,g)$ when $g$ is small. Equations \eqref{eq:tp(c,1)}, \eqref{eq:tp(c,2)}, and \eqref{eq:tp(c,3)} below can be proved via induction using Lemma \ref{lemma:tprecur}. Similarly, one can prove Equation \eqref{eq:t(c,1)} using Lemma \ref{lemma:trecur}. Equation \eqref{eq:tbar(c,1)} follows from Equations \eqref{eq:tp(c,1)} and \eqref{eq:t(c,1)}. For all $c\geq 3$, we have
\begin{subequations}
\label{eq:optim}
\begin{align}
    t_p(c,1) = &\; \begin{cases} 0&\text{if $c\equiv 1$ or $2$ mod $4$},\\
    1 & \text{if $c\equiv 0$ or $3$ mod $4$,}\\
    \end{cases}  
    \label{eq:tp(c,1)}\\
    t_p(c,2) = & \; \left\lfloor\frac{c-1}{4}\right\rfloor,
    \label{eq:tp(c,2)}\\
    t_p(c,3) = & \; \frac{1}{8}\left(2\left\lfloor\frac{c-5}{2}\right\rfloor \left\lfloor\frac{c-1}{2}\right\rfloor - (-1)^{\left\lfloor\frac{c-1}{2}\right\rfloor}+1\right),
    \label{eq:tp(c,3)}\\
    t(c,1) = & \; \left\lfloor\frac{c-1}{2}\right\rfloor,
    \label{eq:t(c,1)}\\
    \bar{t}(c,1) = & \; \left\lfloor\frac{c+1}{4}\right\rfloor.
    \label{eq:tbar(c,1)}
\end{align}
\end{subequations}

\section{Median and Mode}
\label{sec:Mode}

In Section \ref{sec:recursions}, we found recursive and explicit formulas for $t(c,g)$ and $t_p(c,g)$. In this section, we use these to prove that the median and mode of the distribution of genera of 2-bridge knots with crossing number $c$ are both $\left\lfloor\frac{c+2}{4}\right\rfloor$. To do this, we first will prove analogous results for $T(c)$ and $T_p(c)$.

For a word $\omega$ in $T(c)$ or $T_p(c)$, let $K_{\omega}$ be the $2$-bridge knot associated with $\omega$. Define two random variables $G_{T(c)}:T(c)\to\R$ and $G_{T_p(c)}:T_p(c)\to\R$ by $G_{T(c)}(\omega)=g(K_{\omega})$ and $G_{T_p(c)}(\omega) = g(K_{\omega})$, respectively. In other words, $G_{T(c)}$ and $G_{T_p(c)}$ are the random variables that input a word from $T(c)$ and $T_p(c)$, respectively, and output the genus of the knot associated with the word. Recall that $\mathcal{K}_c$ is the set of 2-bridge knots of crossing number $c$ and $G_c:\mathcal{K}_c\to\R$ is defined by $G_c(K)=g(K)$. For each $c$, the random variables $G_c$, $G_{T(c)}$ and $G_{T_p(c)}$ all have a special property that we name quasi-symmetric and define here.

\newpage

\begin{definition}
The sequence $(a_1,a_2,\dots,a_n)$ of integers is \textit{left-dominated quasi-symmetric} if
\[a_{n-j+1}\leq a_{j} \leq a_{n-j}\]
for $1\leq j \leq \left\lfloor \frac{n}{2}\right\rfloor$ and is \textit{right-dominated quasi-symmetric} if
\[a_j\leq a_{n-j+1}\leq a_{j+1}\]
for $1 \leq j \leq \left\lfloor \frac{n}{2}\right\rfloor$. If $(a_1,\dots,a_n)$ is either left- or right-dominated quasi-symmetric, we say the sequence is \textit{quasi-symmetric}. 

Let $X:\Omega\to\R$ be a random variable, where $\Omega$ is a finite set and the range of $X$ is a subset of $\{1,\dots,n\}$. For each $j=1,\dots,n$, define $a_j=|\{\omega\in\Omega~:~X(\omega)=j\}|$. We say that the random variable $X$ is \textit{left-dominated quasi-symmetric}, \textit{right-dominated quasi-symmetric}, or \textit{quasi-symmetric} if the sequence $(a_1,\dots,a_n)$ is left-dominated quasi-symmetric, right-dominated quasi-symmetric, or quasi-symmetric, respectively.
    \end{definition}
As a slight abuse of notation, we write $X:\Omega\to\{1,\dots,n\}$ for a real-valued random variable whose range is a subset of $\{1,\dots,n\}$. A quasi-symmetric sequence is totally ordered in the following sense. If $(a_1,\dots,a_{2k})$ is left-dominated quasi-symmetric, then 
\[a_{2k}\leq a_1\leq a_{2k-1} \leq a_2 \leq \cdots \leq a_{k+2} \leq a_{k-1}\leq a_{k+1}\leq a_{k},\] and if $(a_1,\dots,a_{2k})$ is right-dominated quasi-symmetric, then \[a_1\leq a_{2k}\leq a_2 \leq a_{2k-1}\leq \cdots\leq a_{k-1} \leq a_{k+2}\leq a_{k}\leq a_{k+1}.\] Similarly, if $(a_1,\dots,a_{2k-1})$  is left-dominated quasi-symmetric, then \[a_{2k-1}\leq a_1\leq a_{2k-2} \leq a_2 \leq \cdots \leq a_{k-2} \leq a_{k+1} \leq a_{k-1}\leq a_k,\] and if $(a_1,\dots,a_{2k-1})$  is right-dominated quasi-symmetric, then \[a_1 \leq a_{2k-1}\leq a_2\leq a_{2k-2}\leq \cdots a_{k-2} \leq a_{k-1} \leq a_{k+1}\leq a_k.\]

Table \ref{tab:tbar} indicates that if $3\leq c \leq 20$, then $G_c$ is left-dominated quasi-symmetric when $c$ is odd and right-dominated quasi-symmetric when $c$ is even. The proof below of part (1) of Theorem \ref{thm:main} shows that this fact is true in general.  It will use the following observation.

\begin{remark}
\label{rem:mode}
If $n$ is odd, then the mode of a quasi-symmetric random variable $X:\Omega\to\{1,\dots,n\}$ is $\left\lceil\frac{n}{2}\right\rceil = \frac{n+1}{2}$. If $n$ is even, the mode of a left-dominated quasi-symmetric random variable $X:\Omega\to\{1,\dots,n\}$ is $\frac{n}{2}$, and the mode of a right-dominated quasi-symmetric random variable $X:\Omega\to\{1,\dots,n\}$ is $\frac{n}{2}+1$.
\end{remark}

\begin{lemma}
\label{lemma:m=M}
Let $X:\Omega\to\{1,\dots,n\}$ be a quasi-symmetric random variable. Then the median and mode of $X$ are the same. 
\begin{proof}
For $1\leq j \leq n$, define $a_j=|\{\omega\in\Omega~:~X(\omega) = j\}|$.  Suppose that $n=2k-1$ is odd; the case where $n$ is even is similar. Since $n$ is odd, by Remark \ref{rem:mode} the mode is $k$. The median of $X$ is $k$ if both $P(X\leq k)\geq 1/2$ and $P(X \geq k)\geq 1/2$. Since 
\[P(X\leq k) = \frac{\sum_{j=1}^k a_j}{\sum_{j=1}^n a_j}~\text{and}~P(X\geq k) = \frac{\sum_{j=k+1}^n a_j}{\sum_{j=1}^n a_j},\]
to show that $k$ is also the median, it suffices to show that
\begin{equation}
    \label{ineq:qs}
a_{k} \geq \left| \sum_{j=1}^{k-1} a_j - \sum_{j=k+1}^{n} a_j\right|.
\end{equation}

Suppose that $X$ is left-dominated quasi-symmetric. Thus 
\[a_{2k-1}\leq a_1 \leq a_{2k-2} \leq a_2 \leq \cdots \leq a_{k+1}\leq a_{k-1}\leq a_k.\] 
Therefore by grouping around the first, third, and, more generally, odd-numbered inequalities,
\begin{align*}
    \left| \sum_{j=1}^{k-1} a_j - \sum_{j=k+1}^{n} a_j\right| = & \; |(a_1-a_{2k-1}) + (a_2-a_{2k-2}) + \cdots + (a_{k-1}-a_{k+1})|\\
    = & \; (a_1-a_{2k-1}) + (a_2-a_{2k-2}) + \cdots + (a_{k-1}-a_{k+1})\\
    = & \;  \sum_{j=1}^{k-1} a_j - \sum_{j=k+1}^{n} a_j,
\end{align*}
and by grouping around the second, fourth, and, more generally, even-numbered inequalities,
\[a_{2k-1} + (a_{2k-2}-a_1) + \cdots + (a_{k+1} - a_{k-2}) + (a_k-a_{k-1}) \geq 0,\]
as it is the sum of non-negative terms. Hence Inequality \eqref{ineq:qs} holds, and $k$ is the median of $X$.

Now suppose that $X$ is right-dominated quasi-symmetric. Thus 
\[a_1\leq a_{2k-1}\leq a_2 \leq a_{2k-2}\leq \cdots \leq a_{k-1}\leq a_{k+1}\leq a_k.\]
Therefore by grouping around the first, third, and, more generally, odd-numbered inequalities,
\begin{align*}
    \left| \sum_{j=1}^{k-1} a_j - \sum_{j=k+1}^{n} a_j\right| = & \; |(a_{2k-1}-a_1) + (a_{2k-2}-a_2) + \cdots + (a_{k+1}-a_{k-1})| \\
 = & \;  (a_{2k-1}-a_1) + (a_{2k-2}-a_2) + \cdots + (a_{k+1}-a_{k-1})\\
    = & \; \sum_{j=k+1}^{n} a_j - \sum_{j=1}^{k-1} a_j,
\end{align*}
and by grouping around the second, fourth, and, more generally, even-numbered inequalities,
\[a_1 + (a_2 - a_{2k-1}) + \cdots + (a_{k-1}-a_{k-2}) + (a_k - a_{k+1}) \geq 0,\]
as it is the sum of non-negative terms. Hence Inequality \eqref{ineq:qs} holds, and $k$ is the median of $X$.
\end{proof}
\end{lemma}

Lemmas \ref{lemma:tpmode} and \ref{lemma:tmode} below show that $G_{T_p(c)}$ and $G_{T(c)}$ are quasi-symmetric. Thus Remark \ref{rem:mode} and Lemma \ref{lemma:m=M} give the median and mode of $G_{T_p(c)}$ and $G_{T(c)}$, leading to the proof of part (1) of Theorem \ref{thm:main} to conclude this section.

\begin{lemma}
\label{lemma:tpmode}
The random variable $G_{T_p(c)}:T_p(c)\to\{1,2,\dots, \left\lfloor\frac{c-1}{2}\right\rfloor\}$ is right-dominated quasi-symmetric. Therefore the median and mode of $G_{T_p(c)}$ are both $\left\lfloor \frac{c+3}{4}\right\rfloor$ when $c\equiv 1$ or $2$ mod $4$ and $\left\lfloor \frac{c+1}{4}\right\rfloor$ when $c\equiv 3$ or $0$ mod $4$.
\end{lemma}

\begin{proof} Corollary \ref{t_pevenandoddsame} implies that we only need to prove the statement for $c$ even or $c$ odd. We proceed by induction on  odd $c$. The two base cases are when $c=3$ and $c=5$. These hold because $t_p(3,g)=1$ when $g=1$ and $0$ when $g\neq 1$, and $t_p(5,g)=1$ when $g=2$ and $0$ when $g\neq 2$.

For the inductive step, we assume $c\equiv 1$ mod $4$, but the case where $c\equiv 3$ mod $4$ proceeds similarly. Let $c=4k+1$, and so $\left\lfloor\frac{c-1}{2}\right\rfloor=2k$. In order to show that $G_{T_p(c)}$ is right-dominated quasi-symmetric, we prove for $1\leq j \leq k$ that
\begin{equation}
t_p(c,j) \leq t_p\left(c,2k-j+1\right) \leq t_p(c,j+1).
\label{ineq:qsproof}
\end{equation}

By the inductive hypothesis, the sequence $\left(t_p(c-2,1), t_p(c-2,2),\dots, t_p\left(c-2,2k-1\right)\right)$
is right-dominated quasi-symmetric, and thus
\[t_p(c-2,j) \leq t_p\left(c-2,2k-j\right) \leq t_p(c-2,j+1)\]
for $1\leq j \leq k-1$. 
If $j=1$, then Equation \eqref{eq:tp(c,1)} implies $t_p(c,1)=0$ because $c\equiv 1$ mod $4$. Hence $t_p(c,1)\leq t_p\left(c,2k\right)$.
Suppose that $1<j\leq k-1$. Then  by part (1) of Theorem \ref{thm:3formulas} and the inductive hypothesis
\begin{align*}
    t_p(c,j) = & \; t_p(c-2,j) + t_p(c-2,j-1)\\
    \leq & \; t_p\left(c-2,2k-j\right) +t_p\left(c-2,2k-j+1\right)\\
    = & \; t_p\left(c,2k-j+1\right).
\end{align*}
Similarly,
\begin{align*}
    t_p(c,j+1) = & \; t_p(c-2,j+1) + t_p(c-2,j)\\
    \geq & \; t_p\left(c-2,2k-j\right) +t_p\left(c-2,2k-j+1\right)\\
    = & \; t_p\left(c,2k-j+1\right).
\end{align*}

It remains to show Inequality \eqref{ineq:qsproof} holds when $j=k$. In that case, the two inequalities become
\[t_p(c,k) \leq t_p(c,k+1)\leq t_p(c,k+1),\]
where the latter one obviously holds. The former one holds because
\[ t_p(c,k) = t_p(c-2,k-1)+t_p(c-2,k) \leq t_p(c-2,k+1)+t_p(c-2,k) = t_p(c,k+1).\]
\end{proof}

\begin{lemma}
\label{lemma:tmode}
    The random variable $G_{T(c)}:T(c) \to \left\{1,\dots,\left\lfloor \frac{c-1}{2}\right\rfloor\right\}$ is left-dominated quasi-symmetric when $c$ is odd and right-dominated quasi-symmetric when $c$ is even. Thus the median and mode of $G_{T(c)}$ are both $\left\lfloor\frac{c+2}{4}\right\rfloor$.
\end{lemma}

\begin{proof}
 Lemma \ref{lemma:tpmode} implies that $G_{T_p(2c)}$ is right-dominated quasi-symmetric, and thus there is a total ordering on the quantities $t_p(2c,g)$ for $1\leq g \leq \left\lfloor \frac{2c-1}{2}\right\rfloor$. Since Theorem \ref{thm:3formulas} states that $t_p(2c,2g) = t(c,g)$, the total ordering of $t_p(2c,g)$ induces a total ordering of $t(c,g)$ for $1\leq g \leq \left\lfloor\frac{c-1}{2}\right\rfloor$. The total ordering on the quantities $t(c,g)$ implies that $T(c)$ is left-dominated quasi-symmetric when $c$ is odd and right-dominated quasi-symmetric when $c$ is even. Remark \ref{rem:mode} and Lemma \ref{lemma:m=M} imply that the median and mode of $G_{T(c)}$ are both $\left\lfloor\frac{c+2}{4}\right\rfloor$.
\end{proof}

Since $\bar t(c,g) = \frac{1}{2}\left(t(c,g) + t_p(c,g)\right)$, we use the the facts that $G_{T(c)}$ and $G_{T_p(c)}$ are quasi-symmetric to prove that $G_c$ is quasi-symmetric, yielding formulas for the median and mode of $G_c$. Our proof also uses part (3) of Theorem \ref{thm:3formulas}. 

\begin{proof}[Proof of Theorem \ref{thm:main} part (1)]
For $c\equiv 0, 2, $ and $3$ mod $4$, the mode and median of $G_{T(c)}$ and $G_{T_p(c)}$ both occur at $g= \left\lfloor\frac{c+2}{4}\right\rfloor$. It follows that the mode and median of $G_c$ are also at $g= \left\lfloor\frac{c+2}{4}\right\rfloor$ as $\bar t(c,g) = \frac{1}{2}\left(t(c,g) + t_p(c,g)\right)$. For $c\equiv 1$ mod $4$, the mode and median occur at $g=\frac{c+3}{4}$ for $G_{T_p(c)}$ and at $g = \frac{c-1}{4}=\left\lfloor\frac{c+2}{4}\right\rfloor$ for $G_{T(c)}$. Since the set $T(c)$ is exponentially larger than the set $T_p(c)$, one expects that the median of $G_c$ and $G_{T(c)}$ agree. To prove the median and mode of $G_c$ when $c\equiv 1$ mod $4$ occurs at $g = \frac{c-1}{4}$, we show $G_c$ is left-dominated quasi-symmetric when $c \equiv 1,3$ mod $4$ by induction then apply Lemma \ref{lemma:m=M}.  

The two base cases of $c = 3$ and $c = 5$ can be verified by direct computation. For our inductive step, we assume $c \equiv 1$ mod $4$, but the case where $c \equiv 3$ mod $4$ proceeds similarly.

Let $c = 4k + 1$, and so $\left\lfloor\frac{c-1}{2}\right\rfloor=2k$.  In order to prove $G_c$ is left-dominated quasi-symmetric we prove for $1\leq j \leq k$ that
\begin{equation}
\bar t(c,2k-j+1) \leq \bar t\left(c,j\right) \leq \bar t(c,2k-j).
\label{ineq:tbar}
\end{equation}

By the inductive hypothesis, the sequence $\left(\bar t(c-2,1), \bar t(c-2,2),\dots, \bar t\left(c-2,2k-1\right)\right)$ is left-dominated quasi-symmetric, and thus
\[ \bar t(c-2,2k-j) \leq \bar t\left(c-2,j\right) \leq \bar t(c-2,2k-j-1) \]
for $1\leq j \leq k-1$. 

For for $g=1$, we check $\bar t(c,2k) \leq \bar t(c,1) \leq \bar t(c,2k-1)$. Since $\bar t(c,2k) = \bar t(c-2,2k) + \bar t(c-2,2k-1) + t_p(2c-4,4k-1)$ and $t_p(2c-4,4k-1) = \bar t(c-2,2k) = 0$, a simple inductive argument shows $\bar t(c,2k) = 1$ for all $c$. Equation \eqref{eq:tbar(c,1)} states $\bar t(c,1) =\left \lfloor \frac{c+1}{4} \right \rfloor$, so $1=\bar t(c,2k) \leq \bar t(c,1) = \left \lfloor \frac{c+1}{4}\right \rfloor$. For the second inequality,

\begin{align*}
    \bar t(c,2k-1) = & \; \bar t(c-2,2k-1) + \bar t(c-2,2k-2) + t_p(2c-4, 4k-3)\\
    \geq & \; \bar t(c-2,2k-1) + \bar t(c-2, 1) \\
= & \; 1 + \left\lfloor \frac{c-1}{4} \right\rfloor \geq  \left\lfloor \frac{c+1}{4} \right\rfloor =  \bar t(c,1).
\end{align*}

Now suppose that $1<j\leq k-1$. Then by part (3) of Theorem \ref{thm:3formulas} and the inductive hypothesis,
\begin{align*}
    \bar t(c,2k-j+1) = & \; \bar t(c-2,2k-j+1) + \bar t(c-2,2k-j) + t_p(2c-4, 4k-2j + 1)\\
    \leq & \; \bar t(c-2, j-1) + \bar t(c-2,j) + t_p(2c-4, 2j-1)\\
    = & \; \bar t(c,j),
\end{align*}
where $t_p(2c-4, 2j-1) \geq t_p(2c-4, 4k-2j + 1)$ follows from $G_{T_p}$ being right-dominated quasi-symmetric. Similarly,
\begin{align*}
   \bar t(c,2k-j) = & \; \bar t(c-2,2k-j) + \bar t(c,2k-j-1) + t_p(2c-4,4k-2j-1)\\
    \geq & \; \bar t(c-2,j-1) + \bar t(c-2,j) + t_p(2c-4,2j-1)\\
    = & \; \bar t(c,j).
\end{align*}

It remains to show Inequality \eqref{ineq:tbar} holds when $j=k$. In that case, the two inequalities become
\[\bar t(c,j+1) \leq \bar t(c,j) \leq \bar t(c,j), \]
where the latter one obviously holds. The former one holds because
\begin{align*}\bar t(c,j+1) = &\;  \bar t(c-2,j+1) + \bar t(c-2,j) + t_p(2c,-4,2j+1) \\
\leq & \; \bar t(c-2,j-1) + \bar t(c-2,j) + t_p(2c,-4,2j-1)\\ 
= & \; \bar t(c,j).
\end{align*}

The case where $c\equiv3$ mod $4$ is similar.
\end{proof}

\section{Variance}
\label{sec:variance}

This section focuses on finding the variance $\Var(G_c)$ of the genus of $2$-bridge knots. Theorem \ref{thm:variance} gives an exact but unwieldy formula for $\Var(G_c)$, and part (2) of Theorem \ref{thm:main} follows. Our variance computation requires computing the sum of the genera and the sum of the squares of the genera of knots coming from $T(c)$ and $T_p(c)$.

The \textit{total genus} 
$g(c)$ of $T(c)$ is defined as $g(c)=\sum_{w\in T(c)} g(K_w)$ where $K_w$ is the $2$-bridge knot associated with the word $w\in T(c)$. The \textit{total palindromic genus} $g_p(c)$ of $T_p(c)$ is defined as $g_p(c)=\sum_{w\in T_p(c)} g(K_w)$. Both $g(c)$ and $g_p(c)$ have recurrence relations that are useful in finding the explicit formula for variance.
\begin{lemma}
\label{lemma:gc} Let $c\geq 5$.
\begin{enumerate}
    \item The total genus $g(c)$ of $T(c)$ satisfies
\[g(c) = g(c-1) + 2g(c-2) + t(c-2) + t(c-3).\]
\item The total palindromic genus $g_p(c)$ of $T_p(c)$ satisfies
\[g_p(c) = 2g_p(c-2)+t_p(c-2)+(-1)^{\left\lfloor\frac{c+1}{2}\right\rfloor}.\]
\end{enumerate}
\end{lemma}
\begin{proof}
The definitions of $g(c)$ and $t(c,g)$ and the recursive relation for $t(c,g)$ in Lemma \ref{lemma:trecur} imply
\begin{align*}
 g(c) = &\;  \sum_{i=1}^{\left\lfloor \frac{c-1}{2} \right\rfloor}i t(c,i)\\
 = & \; \sum_{i = 1}^{\left\lfloor \frac{c-1}{2} \right\rfloor} i(t(c-1,i) + t(c-2,i) + t(c-2,i-1) + t(c-3,i-1) - t(c-3,i))\\
 = & \; \sum_{i = 1}^{\left\lfloor \frac{c-1}{2} \right\rfloor} it(c-1,i) + \sum_{i = 1}^{\left\lfloor \frac{c-1}{2} \right\rfloor}it(c-2,i) + \sum_{i = 1}^{\left\lfloor \frac{c-1}{2} \right\rfloor}(i-1)t(c-2,i-1)+ \sum_{i = 1}^{\left\lfloor \frac{c-1}{2} \right\rfloor} t(c-2,i-1)\\
 & \;  + \sum_{i = 1}^{\left\lfloor \frac{c-1}{2} \right\rfloor}t(c-3,i-1) + \sum_{i = 1}^{\left\lfloor \frac{c-1}{2} \right\rfloor}(i-1)t(c-3,i-1)- \sum_{i = 1}^{\left\lfloor \frac{c-1}{2} \right\rfloor} it(c-3,i)\\
 = & \; g(c-1)+2g(c-2) + t(c-2) + t(c-3),
\end{align*}
where the last equality holds because
\begin{align*}
   \sum_{i = 1}^{\left\lfloor \frac{c-1}{2} \right\rfloor}(i-1)t(c-2,i-1) = &\; \sum_{i = 1}^{\left\lfloor \frac{c-3}{2} \right\rfloor} it(c-2,i) = g(c-2)~\text{and}\\
   \sum_{i = 1}^{\left\lfloor \frac{c-1}{2} \right\rfloor}(i-1)t(c-3,i-1)= &\; \sum_{i = 1}^{\left\lfloor \frac{c-3}{2} \right\rfloor} it(c-3,i) = g(c-3).\\
\end{align*}


Equation \eqref{eq:tp(c,1)} imples that $t_p(c,1) =  t_p(c-2,1) +(-1)^{\left\lfloor\frac{c+1}{2}\right\rfloor}$. The definitions of $g_p(c)$ and $t_p(c,g)$ and part (1) of Theorem \ref{thm:3formulas} imply that
\begin{align*}
    g_p(c) = & \; t_p(c,1) + \sum_{i=2}^{\left\lfloor\frac{c-1}{2}\right\rfloor} it_p(c,i)\\
    = & \; t_p(c-2,1) + (-1)^{\left\lfloor\frac{c+1}{2}\right\rfloor}+\sum_{i=2}^{\left\lfloor\frac{c-1}{2}\right\rfloor}it_p(c-2,i)+\sum_{i=2}^{\left\lfloor\frac{c-1}{2}\right\rfloor}it_p(c-2,i-1)\\
    = & \; \sum_{i=1}^{\left\lfloor\frac{c-1}{2}\right\rfloor}it_p(c-2,i)+\sum_{i=1}^{\left\lfloor\frac{c-1}{2}\right\rfloor}(i-1)t_p(c-2,i-1) + \sum_{i=1}^{\left\lfloor\frac{c-1}{2}\right\rfloor}t_p(c-2,i-1) + (-1)^{\left\lfloor\frac{c+1}{2}\right\rfloor}\\
    = & \; 2g_p(c-2)+t_p(c-2)+(-1)^{\left\lfloor\frac{c+1}{2}\right\rfloor}.
\end{align*}
\end{proof}

Define the \textit{total square genus} $g^2(c)$ of $T(c)$ by $g^2(c) = \sum_{w\in T(c)} g(K_w)^2$ where $K_w$ is the $2$-bridge knot associated with the word $w\in T(c)$. Similarly, define the \textit{total square palindromic genus} $g_p^2(c)$ of $T_p(c)$ by $g_p^2(c) = \sum_{w\in T_p(c)} g(K_w)^2$. The quantities $g^2(c)$ and $g^2_p(c)$ are useful when computing variance and have recursive relations as follows. 
\begin{lemma}
\label{lemma:qc}
Let $c\geq 6$.
\begin{enumerate}
    \item The total square genus $g^2(c)$ of $T(c)$ satisfies
    \[g^2(c) = g^2(c-1) + 2g^2(c-2) + g(c) + 2g(c-3) - g(c-1).\]
    \item The total square palindromic genus $g^2_p(c)$ of $T_p(c)$ satisfies
    \[g^2_p(c)=2g^2_p(c-2) + 2g_p(c-2) + t_p(c-2)+(-1)^{\left\lfloor\frac{c+1}{2}\right\rfloor}.\]
\end{enumerate}
\end{lemma}

\begin{proof}
The definitions of $g^2(c)$ and $t(c,g)$ and Lemma \ref{lemma:trecur} imply
\begin{align*}
    g^2(c) = & \; \sum_{i=1}^{\left\lfloor\frac{c-1}{2}\right\rfloor}i^{2}t\left(c,i\right)\\
     = & \; \sum_{i=1}^{\left\lfloor\frac{c-1}{2}\right\rfloor}i^{2}t\left(c-1,i\right) + \sum_{i=1}^{\left\lfloor\frac{c-1}{2}\right\rfloor}i^{2}t\left(c-2,i\right) + \sum_{i=1}^{\left\lfloor\frac{c-1}{2}\right\rfloor}i^{2}t\left(c-2,i-1\right)  \\
     & \; -
\sum_{i=1}^{\left\lfloor\frac{c-1}{2}\right\rfloor}i^{2}t\left(c-3,i\right) + \sum_{i=1}^{\left\lfloor\frac{c-1}{2}\right\rfloor}i^{2}t\left(c-3,i-1\right).\\
\end{align*}

Because $i^2 = (i-1)^2 + 2(i-1) + 1$, it follows that 
\begin{align*}
    \sum_{i=1}^{\left\lfloor\frac{c-1}{2}\right\rfloor}i^{2}t\left(c-2,i-1\right) = & \; \sum_{i=1}^{\left\lfloor\frac{c-1}{2}\right\rfloor} (i-1)^2 t\left(c-2,i-1\right) + 2 \sum_{i=1}^{\left\lfloor\frac{c-1}{2}\right\rfloor} (i-1) t\left(c-2,i-1\right) + \sum_{i=1}^{\left\lfloor\frac{c-1}{2}\right\rfloor} t\left(c-2,i-1\right)\\
    = & \; g^2(c-2) + 2g(c-2) + t(c-2). \\
\end{align*}
Similarly, 
\[\sum_{i=1}^{\left\lfloor\frac{c-1}{2}\right\rfloor}i^{2}t\left(c-3,i-1\right) = g^2(c-3) +2g(c-3) + t(c-3).\]
Therefore
\begin{align*}
g^2(c) = &\;  g^2(c-1) + g^2(c-2) + g^2(c-2) + 2g(c-2) + t(c-2)\\
& \; -g^2(c-3) + g^2(c-3) +2g(c-3) + t(c-3)\\
=& \; g^2(c-1) + 2g^2(c-2) + g(c) + 2g(c-3) - g(c-1),
\end{align*}
where the final equality follows from part (1) of Lemma \ref{lemma:gc}. This yields the recursive formula for $g^2(c)$ in part (1) of the lemma.

The definitions of $g^2_p(c)$, $t_p(c,g)$ and part (1) of Theorem \ref{thm:3formulas} imply that
\begin{align*}
    g_p^2(c) = & \; t_p(c,1) +  \sum_{i=2}^{\left\lfloor\frac{c-1}{2}\right\rfloor} i^{2}t_p(c,i)\\
    = & \; t_p(c-2,1) + (-1)^{\left\lfloor\frac{c+1}{2}\right\rfloor}+\sum_{i=2}^{\left\lfloor\frac{c-1}{2}\right\rfloor} i^{2}t_p(c-2,i) + \sum_{i=2}^{\left\lfloor\frac{c-1}{2}\right\rfloor} i^{2}t_p(c-2,i-1)\\
    = & \; (-1)^{\left\lfloor\frac{c+1}{2}\right\rfloor} + g_p^2(c-2) + \sum_{i=2}^{\left\lfloor\frac{c-1}{2}\right\rfloor} (i-1)^2 t_p\left(c-2,i-1\right)\\
    & + 2 \sum_{i=2}^{\left\lfloor\frac{c-1}{2}\right\rfloor} (i-1) t_p\left(c-2,i-1\right) + \sum_{i=2}^{\left\lfloor\frac{c-1}{2}\right\rfloor} t_p\left(c-2,i-1\right)\\
    = & \;  (-1)^{\left\lfloor\frac{c+1}{2}\right\rfloor} + 2g_p^2(c-2) + 2g_p(c-2) + t_p(c-2),
\end{align*}
as desired.
\end{proof}

The recursive formulas in Lemmas \ref{lemma:gc} and \ref{lemma:qc} yield the following explicit formulas for $g(c)$, $g_p(c)$, $g^2(c)$, and $g^2_p(c)$.

\begin{theorem}
\label{thm:explicit}
    Let $c\geq 3$. Then
    \begin{align*}
        g(c) = & \; \frac{(9c+3)2^{c-3}-24(-1)^{c}}{54},\\
        g_p(c) = & \; \begin{cases}\displaystyle \frac{(3c+2)2^{\frac{c-4}{2}}+4(-1)^{\frac{c}{2}}}{18} & \text{if $c$ is even},\vspace{1mm} \\
        \displaystyle \frac{(3c+5)2^{\frac{c-3}{2}}+4(-1)^{\frac{c-1}{2}}}{18} & \text{if $c$ is odd},
        \end{cases}\\
        g^2(c) = & \; \frac{(9c^2+15c-16)2^{c-6}-20(-1)^c}{27},~\text{and}\\
        g_p^2(c) = & \; \begin{cases}\displaystyle
\frac{(9c^2+48c-25)2^{\frac{c-3}{2}}-16(-1)^{\frac{c-1}{2}}}{216} & \text{if $c$ is odd,}\vspace{1mm} \\
\displaystyle
\frac{(9c^2+30c-64)2^{\frac{c-4}{2}}-16(-1)^{\frac{c-2}{2}}}{216} & \text{if $c$ is even.}
\end{cases}
    \end{align*}
\end{theorem}

\begin{proof}
Each of the formulas is proved by first showing that it evaluates correctly for some initial values of $c$, and then checking that the formula satisfies the relevant recurrence relation from Lemma \ref{lemma:gc} or \ref{lemma:qc}. We leave this straightforward but tedious task to the reader.
\end{proof}

The formulas from Theorem \ref{thm:explicit} yield a lengthy formula for the variance $\Var(G_c)$. Theorem \ref{thm:variance} implies part (2) of Theorem \ref{thm:main}.
\begin{theorem}
\label{thm:variance}
    The variance $\Var(G_c)$ of the genus random variable has formula $\Var(G_c) = \frac{c}{16} - \frac{17}{144} + \varepsilon(c),$ where $\varepsilon(c) =$ 
    \[\begin{cases}
        \frac{3\left(3c-2^{4}\right)2^{\frac{\left(c-4\right)}{2}}+8\left(3c+20\left(-1\right)^{\left(c-1\right)}+3\right)-2^{\left(\frac{c-2}{2}\right)}}{9\left(2^{\left(c+1\right)}+2^{\frac{\left(c+4\right)}{2}}\right)}-\frac{-2^{\frac{\left(c-4\right)}{2}}-2^{\frac{c+4}{2}}+2^{5}}{9\left(2^{\left(2c-1\right)}+2^{\frac{3c+2}{2}}+2^{\frac{2c+2}{2}}\right)}-\frac{1}{9\left(2^{\left(c+2\right)}+2^{\frac{\left(c+6\right)}{2}}\right)} & \small \text{ if } c \equiv 0~\text{mod}~4,\\
        \frac{-\left(3c+1\right)\left(2\right)^{\frac{\left(-c+9\right)}{2}}-2^{\left(-c+7\right)}}{144}+\frac{\left(144+\left(33c-9\right)2^{\frac{\left(c-3\right)}{2}}\right)}{144\left(2^{\left(c-3\right)}+2^{\frac{\left(c-3\right)}{2}}\right)} & \small \text{ if } c \equiv 1~\text{mod}~ 4, \\ 
        \frac{-9c^{2}+9c\left(2^{\frac{\left(c-4\right)}{2}}\right)-25\left(2^{\frac{\left(c-2\right)}{2}}\right)+75c-170}{144\left(2^{\left(c-3\right)}+2^{\frac{\left(c-4\right)}{2}}-1\right)}+\frac{-2^{\left(c-4\right)}-6c\left(2^{\frac{\left(c-4\right)}{2}}\right)+11\left(2^{\frac{\left(c-2\right)}{2}}\right)-9c^{2}+66c-121}{144\left(2^{\left(c-3\right)}+2^{\frac{\left(c-4\right)}{2}}-1\right)^{2}} & \small \text{ if } c \equiv 2~\text{mod}~ 4, \\ 
        \frac{\left(2^{\left(c-3\right)}+2^{\frac{\left(c-4\right)}{2}}-1\right)\left(\left(9c-50\right)\left(2^{\frac{\left(c-4\right)}{2}}\right)-9c^{2}+75c+150\right)-\left(2^{\frac{\left(c-4\right)}{2}}+3c-11\right)^{2}}{144\left(2^{\left(c-3\right)}+2^{\frac{\left(c-4\right)}{2}}-1\right)^{2}} & \small \text{ if } c \equiv 3~\text{mod}~ 4.
    \end{cases} \]
Thus $\Var(G_c)$ asymptotically approaches $\frac{c}{16}-\frac{17}{144}$.
   
\end{theorem}
\begin{proof}
    Recall that $\mathcal{K}_c$ is the set of 2-bridge knots of crossing number $c$. Ernst and Sumners \cite{ErnSum} (see Theorem \ref{thm:ernstsumners}) gave a formula for $|\mathcal{K}_c|$. The variance $\Var(G_c)$ can be expressed as
 \begin{equation}
 \label{eq:var}
    \Var(G_c) = E(G_c^2) -  E(G_c)^2 =  \frac{g^2(c) + g_p^2(c)}{2|\mathcal{K}_c|} - \left(\frac{g(c) + g_p(c)}{2|\mathcal{K}_c|}\right)^2. 
 \end{equation} 
 One can recover the formula for $\Var(G_c)$ in the statement of the theorem by replacing $g(c)$, $g_p(c)$, $g^2(c)$, and $g_p^2(c)$ in Equation \eqref{eq:var} with their expressions in Theorem \ref{thm:explicit}. The term $\varepsilon(c)$ approaches 0 as $c$ goes to infinity, and thus $\Var(G_c)\to \frac{c}{16}-\frac{17}{144}$ as $c$  goes to infinity.
\end{proof}

\section{Normality}
\label{sec:normal}

In this section, we prove Theorem \ref{thm:normal}, showing that the genus random variable $G_c$ is asymptotically normal. An outline of the proof follows. We use the fact that $t_p(c)$ is exponentially smaller than $t(c)$ to show that the cumulative distribution function $P(G_c\leq x)$ and $P(G_{T(c)}\leq x)$ approach one another. In Lemma \ref{lemma:tctpc}, we use the identity $t_p(2c,2g) =t(c,g)$ from part (2) of Theorem \ref{thm:3formulas} to show that $P\left(G_{T(c)}\leq \frac{x}{2}\right)$ approaches $P(G_{T_p(2c)}\leq x)$. In Lemma \ref{lem:tpbinom}, we show that $P(G_{T_p(c)}\leq x)$ approaches the cumulative distribution function of a binomial random variable. We conclude using the fact that a binomial random variable asymptotically approaches a normal random variable.

In the following lemma, we show that the random variable $G_{T_p(c)}$ approaches a binomial random variable $B_n$, defined by $P(B_n=k) = 2^{-n}\binom{n}{k}$ for $0\leq k \leq n$. To compare the binomial random variable $B_n$ and the random variable $G_{T_p(c)}$, we think of the crossing number $c$ as a function of $n$:  specifically let $c(n)=2n+3$. Similarly, we think of genus $g$ as a function of $k$: specifically let $g(k)=k+1$. With these relations, the range of $G_{T_p(c(n))}$ is $\{1,\dots, \frac{c-1}{2}\} = \{g(0),\dots, g(n)\}$. 

Part (2) of this lemma is the least straightforward ingredient to the proof of Theorem \ref{thm:normal}.
\begin{lemma}
    \label{lem:tpbinom}
    Let $c=c(n)=2n+3$ be a function of $n$ with $n\geq4$ so that $c\geq11$.  Let $g=g(k)=k+1$ be a function of $k$ with $1\leq g \leq \frac{c-1}{2}$.
    \begin{enumerate}
        \item For every $n\geq4$, if $2\leq k \leq n$, then
        \[\frac{t_p(c(n),g(k-1))}{\binom{n}{k-1}} \leq  \frac{t_p(c(n),g(k))}{\binom{n}{k}}.\]
        \item The sum
        \[\sum_{k=0}^n |P(B_n=k) - P(G_{T_p(c(n))} = k)|\]
        goes to zero as $n$ goes to infinity.
    \end{enumerate}
\end{lemma}
\begin{proof}
   To prove part (1), we proceed by induction on $n$. When $n=4$, the base case holds because
\[\frac{t_p(c(4),g(1))}{\binom{4}{1}} = \frac{1}{2} \leq \frac{t_p(c(4),g(2))}{\binom{4}{2}} = \frac{2}{3}\leq \frac{t_p(c(4),g(3))}{\binom{4}{3}} = \frac{3}{4} \leq 
\frac{t_p(c(4),g(4))}{\binom{4}{4}} = 1.\]

Part (1) of Theorem \ref{thm:3formulas} and Pascal's rule imply $\frac{t_p(c(n+1),g(k))}{\binom{n+1}{k}} = \frac{t_p(c(n),g(k)) + t_p(c(n),g(k-1))}{\binom{n}{k}+\binom{n}{k-1}}$.  When $k\geq 3$, the inductive hypothesis implies that 
$\frac{t_p(c(n),g(k-2))}{\binom{n}{k-2}}\leq\frac{t_p(c(n),g(k-1))}{\binom{n}{k-1}} \leq \frac{t_p(c(n),g(k))}{\binom{n}{k}}$.  If $a$, $b$, $c$, $d$, $e$, and $f$ are positive real numbers such that $\frac{a}{c} \leq \frac{b}{d} \leq \frac{e}{f}$, then it follows that $\frac{a}{c}\leq\frac{a+b}{c+d}\leq\frac{b}{d} \leq\frac{b+e}{d+f}\leq\frac{e}{f}$, which gives the inequality below, and therefore 
\begin{align*}
    \frac{t_p(c(n+1),g(k-1))}{\binom{n+1}{k-1}} = & \; \frac{t_p(c(n),g(k-2)) + t_p(c(n),g(k-1))}{\binom{n}{k-2}+\binom{n}{k-1}}\\
    \leq & \; \frac{t_p(c(n),g(k-1)) + t_p(c(n),g(k))}{\binom{n}{k-1}+\binom{n}{k}}\\
    = & \; \frac{t_p(c(n+1),g(k))}{\binom{n+1}{k}}.
\end{align*}

To complete the proof of part (1), it remains to show the desired inequality holds when $k=2$, i.e.
\[\frac{t_p(c(n),2)}{\binom{n}{1}} \leq \frac{t_p(c(n),3)}{\binom{n}{2}}.\]
The above inequality can be obtained using Equations \eqref{eq:tp(c,2)} and \eqref{eq:tp(c,3)}.

Since $\frac{P(G_{T_p(c(n))} = g(k))}{P(B_n = k)} = \frac{t_p(c(n),g(k))}{\binom{n}{k}} \cdot \frac{2^n}{t_p(c(n))}$ and $\frac{2^n}{t_p(c(n))}$ only depends on $n$, part (1) of this Lemma implies that $\frac{P(G_{T_p(c(n))} = g(k))}{P(B_n = k)}$ increases as $k$ increases for $k\geq 1$. Therefore there is an $s\in\N$ such that for all $k$ with $1\leq k \leq s$, $P(G_{T_p(c(n))} = g(k)) \leq P(B_n = k)$, and for all $k > s$, $P(G_{T_p(c(n))} = g(k)) > P(B_n = k)$. Assume that $n=2m+1$ is odd and that $s>m$; the other cases can be proved in a similar fashion. Since $n$ is odd with $c=2n+3=4m+5$, Equation \eqref{eq:tp(c,1)} implies $t_p(c(n),1)=0$, and thus $P(B_n=0) \geq P(G_{T_p(c(n))}=1)$. Define
\begin{align*}
    \Delta_1(n) = & \; \sum_{k=0} ^{m} [P(B_n = k) - P(G_{T_p(c(n))} = g(k))],\\
    \Delta_2(n) = & \; \sum_{k=m+1} ^{s} [P(B_n = k) - P(G_{T_p(c(n))} = g(k))],~\text{and}\\
    \Delta_3(n) = & \; \sum_{k=s +1} ^{n} [P(G_{T_p(c(n))} = g(k)) - P(B_n = k)].
\end{align*}
Then
\[
 \sum_{k=0} ^{n} |P(G_{T_p(c(n))} = g(k)) - P(B_n = k)| = \Delta_1(n) + \Delta_2(n) + \Delta_3(n).\]

By construction, $0\leq \Delta_i(n)$ for $i=1,2,$ and $3$. We prove  that $\sum_{k=0}^n |P(G_{T_p(c(n))} = g(k)) - P(B_n = k)|$ goes to zero as $n$ goes to infinity by showing that $\Delta_2(n)\leq \Delta_1(n)$, $\Delta_3(n)= \Delta_1(n)+\Delta_2(n)$, and $\lim_{n\to\infty} \Delta_1(n) = 0$.


Since $\binom{n}{k} = \binom{n}{n-k}$, it follows that $P(B_n=k) = P(B_n=n-k)$ for $0\leq k \leq m$. Also, Lemma \ref{lemma:tpmode} states that $G_{T_p(c(n))}$ is right-dominated quasi-symmetric, and thus $P(G_{T_p(c(n))} = g(k)) \leq P(G_{T_p(c(n))}=g(n-k))$ for $0\leq k \leq m$. Hence the differences between the probabilities are greater before the median than after, that is, 
\[P(B_n=m-j)-P(G_{T_p(c(n))}=g(m-j))\geq P(B_n=m+1+j)-P(G_{T_p(c(n))}=g(m+1+j))\]
for $j=0,\dots,s-m-1$. Therefore 
\begin{align*}
    \Delta_2(n) = & \; \sum_{k=m+1} ^{s} [P(B_n = k) - P(G_{T_p(c(n))} = g(k))] \\
    \leq & \; \sum_{k=n-s}^m [P(B_n = k) - P(G_{T_p(c(n))} = g(k))]\\
    \leq & \; \sum_{k=0}^m [P(B_n = k) - P(G_{T_p(c(n))} = g(k))] = \Delta_1(n).
\end{align*}

Since the differences of the two probabilities sum to zero $\sum_{k=0}^n [P(G_{T_p(c(n))}=g(k))-P(B_n=k)]=0$ and all $\Delta_i(n)\geq0$, it follows that $\Delta_3(n)=\Delta_1(n)+\Delta_2(n)$.

It remains to show that $\lim_{n\to\infty} \Delta_1(n)=0$. First, observe that $\sum_{k=0}^m P(B_n=k)=\frac{1}{2}$. Lemma \ref{lemma:tpmode} implies that the median of $G_{T_p(c(n))}$ is $g(m+1)$. Therefore both $\frac{1}{2}\geq \sum_{k=0}^m P(G_{T_p(c(n))}=g(k))$ and $\frac{1}{2}\leq \sum_{k=0}^{m+1} P(G_{T_p(c(n))}=g(k))$. The inequality $P(G_{T_p(c(n))}=g(m+1)) \leq P(B_n=m+1)=2^{-n} \binom{n}{m+1}$ and the fact that $\lim_{n\to\infty}P(B_n=m+1)=0$ together imply that $\lim_{n\to\infty} P(G_{T_p(c(n))}=g(m+1))=0$, as well. Therefore $\sum_{k=0}^m P(G_{T_p(c(n))}=g(k))$ approaches $\frac{1}{2}$ as $n$ goes to infinity. Since both $\sum_{k=0}^m P(B_n=k)$ and $\sum_{k=0}^m P(G_{T_p(c(n))}=g(k))$ approach $\frac{1}{2}$ as $n$ goes to infinity, their difference $\Delta_1(n)$ approaches zero as $n$ goes to infinity.

Since $\Delta_2(n)\leq\Delta_1(n)$, we have that $\Delta_2(n)$ approaches zero as $n\to\infty$.  Since $\Delta_3(n)=\Delta_1(n)+\Delta_2(n)$, we have that $\Delta_3(n)$ approaches zero as $n\to\infty$, and the desired result follows.
\end{proof}

In the following lemma, we use the identity $t(c,g) = t_p(2c,2g)$ from part (2) of Theorem \ref{thm:3formulas} to show a relationship between the cumulative distribution functions of $G_{T(c)}$ and $G_{T_p(2c)}$.

\begin{lemma} 
\label{lemma:tctpc}
For all $\ell\in\N$,
    the difference $P(G_{T(c)} \leq \frac{\ell}{2}) - P(G_{T_p(2c)} \leq \ell) $ approaches $0$  as $c$ goes to infinity.
\end{lemma}

\begin{proof}
 Part (2) of Theorem \ref{thm:3formulas} states that $t_p(2c,2g) = t(c,g)$, and thus
\begin{align*}
    P\left(G_{T(c)} \leq \frac{\ell}{2}\right) - P\left(G_{T_p(2c)} \leq \ell\right) & \; = \sum_{g=1}^ {\left\lfloor \frac{\ell}{2} \right\rfloor} \frac{t(c,g)}{t(c)} - \sum_{g=1}^ {\ell} \frac{t_p(2c,g)}{t_p(2c)}\\
   & = \frac{3}{2^{c-2} - (-1)^c} \sum_{g=1}^ {\lfloor \frac{\ell}{2} \rfloor} t_p(2c,2g) - \frac{3}{2^{c-1} - (-1)^{c-1} }\sum_{g=1}^ {\ell} t_p(2c,g)\\
   & = \frac{3}{2^{c-1} - (-1)^{c-1} } \left( \sum_{g=1}^ {\lfloor \frac{\ell}{2} \rfloor} 2t_p(2c,2g) - \sum_{g=1}^ {\ell} t_p(2c,g)\right)\\
   & \; - \frac{3(-1)^c}{(2^{c-1} - (-1)^{c-1}) (2^{c-1} - 2(-1)^c) } \sum_{g=1}^ {\lfloor \frac{\ell}{2} \rfloor} 2t_p(2c,2g). \\
\end{align*}
Since $\sum_{g=1}^ {\lfloor \frac{\ell}{2} \rfloor} 2t_p(2c,2g) \leq \sum_{g=1}^ {\lfloor \frac{\ell}{2} \rfloor} 2t(c,g)   \leq 2t(c)=\frac{2^{c-1} - 2(-1)^c }{3}$, the expression 
\[\frac{3(-1)^c}{(2^{c-1} - (-1)^{c-1}) (2^{c-1} - 2(-1)^c) } \sum_{g=1}^ {\lfloor \frac{\ell}{2} \rfloor} 2t_p(2c,2g)\]
goes to $0$ as $c$ goes to infinity. It remains to show that
\[\frac{3}{2^{c-1} - (-1)^{c-1} } \left( \sum_{g=1}^ {\lfloor \frac{\ell}{2} \rfloor} 2t_p(2c,2g) - \sum_{g=1}^ {\ell} t_p(2c,g)\right)\]
approaches zero as $c$ goes to infinity. 
If $\ell$ is odd, then
\begin{align*}
 \sum_{g=1}^ {\lfloor \frac{\ell}{2} \rfloor} 2t_p(2c,2g) - \sum_{g=1}^ {\ell} t_p(2c,g) = & \; \sum_{g=1}^ {\lfloor \frac{\ell}{2} \rfloor} t_p(2c,2g) - \sum_
{\substack{g=1 \\ g~\text{odd}}}^{\ell} t_p(2c,g)\\
= & \; \left(\sum_{g=1}^{\ell-1} t_p(2c-2,g)\right) - \left(\sum_{g=1}^\ell t_p(2c-2,g)\right) - t_p(2c,1)\\
= & \; -t_p(2c,1)-t_p(2c-2,\ell),
\end{align*}
where the second equality follows from the relation $t_p(c,g) = t_p(c-2,g-1)+t_p(c-2,g)$ for $g\neq 1$. If $\ell$ is even, one can similarly obtain that
\[\sum_{g=1}^ {\lfloor \frac{\ell}{2} \rfloor} 2t_p(2c,2g) - \sum_{g=1}^ {\ell} t_p(2c,g) = -t_p(2c,1).\]
Therefore 
\[\frac{3}{2^{c-1} - (-1)^{c-1} } \left( \sum_{g=1}^ {\lfloor \frac{\ell}{2} \rfloor} 2t_p(2c,2g) - \sum_{g=1}^ {\ell} t_p(2c,g)\right)\]
approaches zero as $c$ goes to infinity, and the result follows.
\end{proof}

Recall that $\Phi_{\mu,\sigma}(x)$ is the cumulative distribution function of a normal random variable with mean $\mu$ and standard deviation $\sigma$. In Lemma \ref{lem:tpbinom}, we compare the random variable $G_{T_p(c)}$ with a binomial random variable $B_n$ where $c=2n+3$. The binomial random variable $B_n$ has mean $\frac{n}{2}$ and standard deviation $\frac{\sqrt{n}}{2}$. The De Moivre-Laplace theorem implies that the cumulative distribution function of $B_n$ approaches $\Phi_{\mu,\sigma}(x)$ with $\mu=\frac{n}{2}$ and $\sigma=\frac{\sqrt{n}}{2}$.  
In the following lemma, we show that the difference between $P(G_{T(c)}\leq x)$ and $\Phi_{\mu,\sigma}(x)$ goes to zero as $c$ goes to infinity.

\begin{lemma}
\label{lemma:tcnormal}
Let $c(n) = 2n+3$, the mean $\mu=\frac{n}{2}$, and the standard deviation $\sigma=\frac{\sqrt{n}}{2}$. For any $x\in\mathbb{R}$, the difference
\[P(G_{T(c(n))}\leq x) - \Phi_{\mu,\sigma}(x)\]
goes to zero as $c$ goes to infinity.
\end{lemma}

\begin{proof} In part (2) of Lemma \ref{lem:tpbinom}, we show that $\sum_{k=0}^n |P(G_{T_p(c(n))} = k)-P(B_n=k)|$ goes to zero as $c$ goes to infinity, and thus $P(G_{T_p(c(n))}\leq \ell) - P(B_n\leq \ell)$ also goes to zero. Since $G_{T_p(c(n))}$ and $B_n$ are discrete random variables with the same range, it follows that $P(G_{T_p(c(n))}\leq x) - P(B_{n}\leq x)$ goes to zero as $n$ goes to infinity for any $x\in\mathbb{R}$. In the following argument, we use the fact that  $P(G_{T_p(c(2n))}\leq 2x) - P(B_{2n}\leq 2x)$ approaches zero as $n$ goes to infinity. 

By the De Moivre-Laplace theorem, a binomial random variable can be approximated via a normal random variable, and thus $P(B_{2n}\leq 2x) - \int_{-\infty}^{2x} \frac{1}{2\sigma\sqrt{2\pi}}e^{\frac{1}{2}\left(\frac{t-2\mu}{2\sigma}\right)^2}\;dt$ approaches zero as $n$ goes to infinity. It follows that
\[P(G_{T_p(c(2n))}\leq 2x) - \int_{-\infty}^{2x} \frac{1}{2\sigma\sqrt{2\pi}}e^{-\frac{1}{2}\left(\frac{t-2\mu}{2\sigma}\right)^2}\;dt\]
approaches zero as $n$ goes to infinity.

Let $t=2w$. Then
\[\int_{-\infty}^{2x} \frac{1}{2\sigma \sqrt{2 \pi}} e^{ -\frac{1}{2}(\frac{t - 2\mu}{2\sigma})^2} \; dt = \int_{-\infty}^{x} \frac{1}{2\sigma \sqrt{2 \pi}} e^{ -\frac{1}{2}(\frac{2w - 2\mu}{2\sigma})^2} 2dw = \int_{-\infty}^{x} \frac{1}{\sigma \sqrt{2 \pi}} e^{ -\frac{1}{2}(\frac{w - \mu}{\sigma})^2}\; dw=\Phi_{\mu,\sigma}(x).\]
Therefore $P(G_{T_p(c(2n))}\leq 2x) - \Phi_{\mu,\sigma}(x)$ approaches zero as $n$ goes to infinity. Lemma \ref{lemma:tctpc} implies that $P(G_{T(c)}\leq x) -\Phi_{\mu,\sigma}(x)$ approaches zero as $n$ goes to infinity, as well.
\end{proof}

We conclude the paper with the proof of Theorem \ref{thm:normal}.
\begin{proof}[Proof of Theorem \ref{thm:normal}]

For any $x\in\mathbb{R}$, the difference of the cumulative distribution functions satisfies
\[  P(G_c\leq x) - \Phi_{\mu,\sigma}(x)
 =  \left(P(G_c\leq x) - P(G_{T(c)}\leq x)\right) + \left( P(G_{T(c)}\leq x) - \Phi_{\mu,\sigma}(x)\right).
\]

Lemma \ref{lemma:tcnormal} implies that the difference $ P(G_{T(c)}\leq x) - \Phi_{\mu,\sigma}(x)$ approaches zero as $c$ approaches infinity. It remains to show that $P(G_c\leq x) - P(G_{T(c)}\leq x)$ also approaches zero.

For any $\ell\in\mathbb{N}$, we have
\begin{align*}
   | P(G_c\leq \ell) - P(G_{T(c)}\leq \ell)| = & \; \left|\frac{\sum_{g=1}^\ell (t(c,g)+t_p(c,g))}{t(c)+t_p(c)} - \frac{\sum_{g=1}^\ell t(c,g)}{t(c)}\right|\\
    = & \; \frac{\left|t(c)\sum_{g=1}^\ell t(c,g) - t_p(c)\sum_{g=1}^\ell t(c,g)\right|}{t(c) ( t(c)+t_p(c))}\\
    \leq & \; \frac{t(c)\sum_{g=1}^\ell t(c,g) + t_p(c)\sum_{g=1}^\ell t(c,g)}{t(c) ( t(c)+t_p(c))}\\
    \leq & \; \frac{2t(c)t_p(c)}{t(c) ( t(c)+t_p(c))}\\
    = & \; \frac{2t_p(c)}{t(c)+t_p(c)}.
\end{align*}
Equations \eqref{eq:tc} and \eqref{eq:tpc} imply that $\frac{2t_p(c)}{t(c)+t_p(c)}$ approaches zero as $c$ goes to infinity. Because the ranges of both $G_c$ and $G_{T(c)}$ are subsets of $\mathbb{N}$, it follows that $P(G_c\leq x) - P(G_{T(c)}\leq x)$ approaches zero for any $x\in\mathbb{R}$.
\end{proof}

\bibliographystyle{amsalpha}
\bibliography{Dist}

\end{document}